\documentclass[12pt,fleqn]{article}

\usepackage[utf8]{inputenc}
\usepackage[T1]{fontenc}

\usepackage{mathtools}
\usepackage{amsfonts,amsmath, amsthm,amssymb,esint}

\allowdisplaybreaks

\usepackage{Alegreya}

\usepackage{authblk}

\usepackage{color}
\definecolor{DB}{rgb}{0.0,0.0,0.8} 
\definecolor{DG}{rgb}{0.0,0.55,0.14}
\definecolor{DR}{rgb}{0.5,0,0.07}

\usepackage[hyperindex=true,colorlinks=true, urlcolor={DB}, linkcolor={DB}, menucolor={DB}, citecolor={DG}, anchorcolor={DG},linktoc=all]{hyperref}

\usepackage{url}

\usepackage{amsrefs}

\usepackage{bbding}

\newcommand{\verti}[1]{{\left\vert #1 \right\vert}}
\newcommand{\vertii}[1]{{\left\vert\kern-0.25ex\left\vert #1 \right\vert\kern-0.25ex\right\vert}}  
\newcommand{\vertiii}[1]{{\left\vert\kern-0.25ex\left\vert\kern-0.25ex\left\vert #1 \right\vert\kern-0.25ex\right\vert\kern-0.25ex\right\vert}}

\def\Lip{\operatorname{Lip}}

\def\bt{\begin{theo}}
\def\et{\end{theo}}
\def\bpr{\begin{proposition}}
\def\epr{\end{proposition}}
\def\bl{\begin{lemma}}
\def\el{\end{lemma}}
\def\bc{\begin{coro}}
\def\ec{\end{coro}}
\def\br{\begin{rema}}
\def\er{\end{rema}}
\def\bp{\begin{proof}}
\def\ep{\end{proof}}

\numberwithin{equation}{section}

\def\Jac{\operatorname{Jac}}

\usepackage{mathrsfs}
\usepackage{stmaryrd}

\usepackage{graphicx}

\usepackage[margin=25mm]{geometry}

\def\R{{\mathbb R}}
\def\N{{\mathbb N}}
\def\so{{\mathbb S}^1}
\def\C{{\mathbb C}}
\def\Z{{\mathbb Z}}
\def\D{{\mathbb D}}
\def\L{{\mathcal L}}
\def\p{\partial}
\def\fo{\forall\,}
\def\l{\label}
\def\bes{\begin{equation*}}
\def\ees{\end{equation*}}
\def\be{\begin{equation}}
\def\ee{\end{equation}}
\def\ba{\begin{aligned}}
\def\ea{\end{aligned}}
\def\d{\displaystyle}
\def\na{\nabla}
\def\O{\Omega}
\def\ve{\varepsilon}
\def\im{\imath}
\def\va{\varphi}

\def\curl{\operatorname{curl}}
\def\dist{\operatorname{dist}}

\renewcommand{\div}{\operatorname{div}}

\usepackage{enumerate}
\def\ben{\begin{enumerate}}
\def\een{\end{enumerate}}

\theoremstyle{definition}
\newtheorem{prop}{Proposition}[section]
\newtheorem{theo}[prop]{Theorem}
\newtheorem{coro}[prop]{Corollary}
\newtheorem{lemma}[prop]{Lemma}
\theoremstyle{definition}

\theoremstyle{definition}
\newtheorem{rema}[prop]{Remark}
\theoremstyle{definition}

\def\ep{\end{proof}}
\def\bp{\begin{proof}}

\newcommand\blfootnote[1]{%
  \begingroup
  \renewcommand\thefootnote{}\footnote{#1}%
  \addtocounter{footnote}{-1}%
  \endgroup
}

\usepackage{csquotes}

\title{Towards an asymptotic analysis of the anisotropic Ginzburg-Landau model}

\author{Dmitry Golovaty, Petru Mironescu and Peter Sternberg}

\date{\today}

\begin{document}

\maketitle

{\centering\footnotesize To the memory of Ha\"\i m Brezis, a friend and a mentor\par}

\begin{abstract}
We develop a set of tools for the asymptotic analysis of minimizers of the anisotropic Ginzburg-Landau energy functional among the admissible competitors with Dirichlet boundary datum of negative degree $-D$. As a byproduct of our analysis, we prove that the energy of a minimizer is $K\ln (1/\varepsilon)+o(\ln (1/\varepsilon))$, where $K$ depends only on $D$ and the material constants that enter into the expression for the energy.
\end{abstract}

\section{Introduction}
\blfootnote{DG: Department of Mathematics, University of Akron, Akron, OH 44325-4002, USA. \url{dmitry@uakron.edu}}
\blfootnote{PM: Universite Claude Bernard Lyon 1, CNRS, Ecole Centrale de Lyon, INSA Lyon, Université Jean Monnet, ICJ UMR5208, 69622 Villeurbanne, France.  
\url{mironescu@math.univ-lyon1.fr}}
\blfootnote{PS: Department of Mathematics, Indiana University, 
Bloomington, IN 47405-7106, USA \url{sternber@iu.edu}}
\blfootnote{Keywords: Anisotropic Ginzburg-Landau, Oseen-Frank, Nematic, Vortex, Pohozaev Identity}
\blfootnote{MSC 2020 classification: 35B25, 35B40, 35J20}

Minimizing the Ginzburg-Landau energy in a 2D domain subject to Dirichlet boundary conditions has been well understood since the seminal contribution of Bethuel, Brezis, and H\'elein \cite{bbh}. In particular, there is no distinction between the analysis of minimizers for boundary datum of positive and negative degree as the two cases are related by conjugation. However, somewhat surprisingly, when the Dirichlet integral is broken into the sum of the squares of the divergence and curl with arbitrary positive weights, the distinction arises. Such a decomposition of the gradient is not merely an academic exercise as it arises in modeling of nematic liquid crystals, in particular within the context of the Oseen-Frank model for uniaxial nematics, see \cite{virga}. 

\smallskip
The case of positive degree, considered by Colbert-Kelly and Phillips in \cite{phillips}, is reducible to the standard treatment, following the ideas of \cite{bbh}. This reduction relies on existence of degree one singularities with bounded energy that are purely divergence or purely curl, and there do not exist analogous vector fields of negative degree. Other interesting cases include extreme situations of high anisotropy when the ratio of the elastic constants is vanishingly small. For example, Golovaty, Sternberg, and Venkatraman \cite{GSV} show that limiting configurations may exhibit line singularities accommodating high deformation cost associated with divergence through the emergence of jumps in the tangential component. In a related work, Kowalczyk, Lamy, and Smyrnelis \cite{kowalczyk} construct entire solutions of the Euler-Lagrange equations having negative degree and possessing equivariant symmetry. 

\smallskip
In this paper, we develop some tools that we believe should be helpful in completing the asymptotic analysis of the minimization problem for the anisotropic Ginzburg-Landau functional, subject to boundary datum of negative degree. We begin by providing the precise statement of the problem and establishing some necessary notation.

\subsection{The problem}
\l{s1}
We let
$\O\subset\R^2\sim\C$ be a smooth bounded domain that we assume to be simply connected. For technical reasons, we occasionally also assume that $\O$ is strictly star-shaped. For $K_1, K_3>0$, $\omega\subset\O$ and $u:\omega\to\C$, we consider the energies
\begin{gather*}
E_0(u)=E_0(u, \omega):=\frac{K_1}2\int_\omega (\div u)^2+\frac{K_3}2\int_\omega (\curl u)^2
\end{gather*}
and
\begin{gather*}
E_\ve(u)=E_\ve(u, \omega):=E_0(u)+\frac 1{4\ve^2}\int_\omega(1-|u|^2)^2,
\end{gather*}
where $\ve>0$. With no loss of generality, we assume that $K_1+K_3=2$. Noting that 
\bes
\frac{K_1}2 (\div u)^2+\frac{K_3}2(\curl u)^2=\frac{K_3}2 (\div \imath u)^2+\frac{K_1}2 (\curl \imath u)^2,
\ees
we may assume in the analysis below that $K_1\ge K_3$, so that we can write
\be
\l{aa1}
K_1=1+\delta,\ K_3=1-\delta,\ \text{with }0\le \delta<1.
\ee

For $u:\omega\to\C$, we denote by
\begin{gather*}
G_0(u)=G_0(u, \omega):=\frac 12\int_\omega |\na u|^2,
\end{gather*}
and
\begin{gather*}
G_\ve(u)=G_\ve(u,\omega):=G_0(u)+\frac 1{4\ve^2}\int_\omega (1-|u|^2)^2,
\end{gather*}
the standard Dirichlet and Ginzburg-Landau energies, respectively. When $\delta=0,$ the functional $E_0$ reduces to $G_0$ while the functional $E_\ve$ reduces to $G_\ve$ (after integration by parts and modulo a fixed boundary term; see the proof of Lemma \ref{af1}).
In what follows, we denote by lower case letters the energy densities, e.g.,
\begin{gather*}
e_0(u):=\frac{K_1}{2}(\div u)^2+\frac {K_3}2 (\curl u)^2,
\end{gather*}
and
\begin{gather*}
g_\ve(u):=\frac 12|\na u|^2+\frac 1{4\ve^2}(1-|u|^2)^2,
\end{gather*}
for $E_0$ and $G_\ve$, respectively.

\smallskip
Given $g:\p\O\to\so$ a smooth map of degree $-D<0$, we let $u_\ve$ denote a minimizer of $E_\ve$ in the class $H^1_g(\O; \C):=\{ u\in H^1(\O ; \C);\, \text{tr}\, u=g\}$. We are interested in the asymptotic properties of $u_\ve$ as $\ve\to 0$, and our main purpose is to extend to $E_\ve$ some of the analysis achieved for $G_\ve$ in \cite{bbh} . 

\subsection{The main results}

Although our results are far from being as complete as those in \cite{bbh}, we feel that they may have some interest and give impetus for subsequent research. Most of the techniques that we use have roots in \cite{bbh} and subsequent works. In particular, several proofs are in the spirit of Struwe \cites{struwe,struwe_2} or Sandier and Serfaty \cite{sandier_serfaty}*{Chapter 5} (see also Han and Shafrir \cite{HS95}, Jerrard \cite{Jer99}, and Sandier \cite{San98}). Part of the analysis consists of establishing {\it a priori} estimates. Such estimates are also obtained for critical points of $E_\ve$, either under energy bounds assumptions or when $\O$ is strictly star-shaped. 
%

\smallskip
A significant part of our analysis is valid for every $K_1$ and $K_3$.  For example, we prove that minimizers $u_\ve$ of $E_\ve$ satisfy, for small $\ve$,  the bounds $|u_\ve|\le C_1$, $|\na u_\ve|\le C_2/\ve$ (Lemma \ref{aq3}). In the case of the standard Ginzburg-Landau equation, this follows from a maximum principle that does not seem to be available in our case. This is derived {\it via} various Pohozaev identities (see, e.g., Lemma \ref{ak3}) and elliptic estimates (see, e.g., Lemma \ref{ah1}).

\smallskip
We establish an $\eta$-ellipticity result (Lemma \ref{ab8}) similar to the one for the standard Ginzburg-Landau equation, asserting, roughly speaking, that if the energy of a minimizer $u_\ve$ is small when compared to $\ln (1/\ve)$, then $u_\ve$ has no vortices. We also prove that critical points of $E_\ve$ satisfying a logarithmic energy bound (and, in particular, minimizers) display a controlled bad discs structure (Lemma \ref{at1}). These bad discs are far away from the boundary (Corollary \ref{aw4}). We also prove the existence of bounded entire local minimizers of negative degree (Corollary \ref{be1}).

\smallskip
Sharper results are established under the assumption that $K_1$ and $K_3$ are \enquote{close}, i.e., for sufficiently small $|\delta|$. For example, wwe prove that, when $|\delta|$ is small, the local minimizers in Corollary \ref{be1} have degree $-1$ (Corollary \ref{be9}). Moreover, when $|\delta|$ and $\ve$ are small, we prove that the bad discs structure associated with a minimizer of $E_\ve$ with respect to a boundary datum of degree $-D<0$ consists of exactly $D$ bad discs, each of degree $-1$ (Theorem \ref{bg1}).

\smallskip
When $|\delta|$ and $\ve$ are small and $0<\alpha<1$, in Section \ref{bm1} we prove that the bad discs are at distance $\ge$ $\ve^\alpha$ from each other and from the boundary (Theorem \ref{bm2}). We complement these results in Section \ref{bu1}, where we are also able to show  that the energy density concentrates on bad disks as $\ve\to0$ (Theorem \ref{bw1}).

\smallskip
Another series of results concerns the energy of minimizers of $E_\ve$ with boundary datum of degree $-D$. For arbitrary $\delta$, we introduce the concept of giant bad discs, that allows us to obtain the asymptotic expansion of this energy up to an $\d o\left(\ln (1/\ve)\right)$ term (Theorem \ref{cb2}). When $|\delta|$ is sufficiently small, we prove that the leading term in the expansion of the energy is $D C_\delta \ln (1/\ve)$. It is well-known that, for the standard Ginzburg-Landau functional investigated in \cite{bbh} and which corresponds to $\delta=0$, we have $C_0=\pi$, and the above term is the leading term for both positive $D$ (as in our work) and negative $D$. When $\delta\neq 0$ and $D$ is negative, it was proved in \cite{phillips}  that the leading order is $D (1-|\delta|)\pi \ln (1/\ve)$. We prove that, when $\delta\neq 0$ is small, the cost of negative degrees is different from the one of positive degrees. More specifically, we prove that, when $|\delta|$ is small, we have $C_\delta>(1-|\delta|)\pi$ (Lemma \ref{br0} and Theorem \ref{ca4}).

\medskip
\noindent
{\bf Acknowlegments.} D.G. acknowledges  support by an NSF grant DMS 2106551. The research of P.S. was supported by a Simons Collaboration grant 585520 and an NSF grant DMS 2106516.

\section{Preliminaries}
\l{s2}
We will repeatedly use the following observations.
\bl
\l{af1}
Let $\omega$ be a bounded Lipschitz domain and $u=v+\imath w\in H^1(\omega ; \C)$. Then 
\be
\l{af2}
\ba
E_0(u, \omega)=&\frac{K_1}2\int_\omega ([v_x]^2+[w_y]^2)+\frac{K_3}2\int_\omega ([v_y]^2+[w_x]^2)\\
&+(K_1-K_3)\int_\omega v_xw_y+K_3\int_\omega (v_xw_y-v_yw_x)
\\
=&\frac{K_1}2\int_\omega ([v_x]^2+[w_y]^2)+\frac{K_3}2\int_\omega ([v_y]^2+[w_x]^2)\\
&+(K_1-K_3)\int_\omega v_xw_y+\frac{K_3}2\int_{\p\omega} u\wedge\frac
{\p u}{\p\tau},
\ea
\ee
and
\be
\l{af3}
\ba
E_0(u, \omega)=&\frac{K_1}2\int_\omega ([v_x]^2+[w_y]^2)+\frac{K_3}2\int_\omega ([v_y]^2+[w_x]^2)\\
&+(K_1-K_3)\int_\omega v_yw_x+K_1\int_\omega (v_xw_y-v_yw_x)
\\
=&\frac{K_1}2\int_\omega ([v_x]^2+[w_y]^2)+\frac{K_3}2\int_\omega ([v_y]^2+[w_x]^2)\\
&+(K_1-K_3)\int_\omega v_yw_x+\frac{K_1}2\int_{\p\omega} u\wedge\frac
{\p u}{\p\tau}.
\ea
\ee

Moreover,
\be
\l{af4}
\ba
(1-\delta)G_0(u,\omega)-\frac{1-\delta}2\bigg|\int_{\p\omega} u\wedge\frac
{\p u}{\p\tau}\bigg|
\le & E_0(u,\omega)
\\
\le & (1+\delta) G_0(u,\omega)+\frac{1+\delta}2\bigg|\int_{\p\omega} u\wedge\frac
{\p u}{\p\tau}\bigg|.
\ea
\ee
\el
\bp
Identities \eqref{af2} and \eqref{af3} are straightforward consequences of 
\be
\l{af5}
\int_\omega (v_xw_y-v_yw_x)=\frac 12\int_{\p\omega} u\wedge\frac
{\p u}{\p\tau}.
\ee
The first and the second inequality in \eqref{af4} follow from the second identity in \eqref{af2} and \eqref{af3}, respectively, once we observe that
\bes
(K_1-K_3)\int_\omega v_xw_y\geq-\delta\int_\omega ([v_x]^2+[w_y]^2)
\ees
and
\bes(K_1-K_3)\int_\omega v_yw_x\leq\delta\int_\omega ([v_y]^2+[w_x]^2).\qedhere
\ees
\ep

%
Next, recalling that $\deg g=-D<0,$ we  prove the following lemma.
\bl
\l{aq1} 
 For small $\ve$, we have the $\delta$-independent bound
\be
\l{aq2}
\min\{ E_\ve(u);\, u\in H^1_g(\O ; \C)\}\le \pi  D \ln \frac 1 \ve+ C(g).
\ee
\el
\begin{proof}
Using  the standard construction of competitors for the Ginzburg-Landau energy, it suffices to prove the result when $\O$ is the unit disc, $D=1$, and $g(z)=\overline z$. Consider, for $0<\ve<1$,  the competitor 
\bes
u(z)=\begin{cases}
\overline z/|z|,&\text{if } |z|\ge\ve
\\
\overline z/\ve,&\text{if }|z|\le \ve
\end{cases}.
\ees

Then 
\bes
E_\ve(u)=\pi \ln \frac 1\ve+ \frac\pi 2\int_0^1 r(1-r^2)^2\, dr.\qedhere
\ees
\end{proof}

The following is straightforward.
\bl
\l{aa2} A critical point 
$u_\ve=v_\ve+\im w_\ve$ of $E_\ve$ in $H^1_g(\O ; \C)$ satisfies
\be
\l{aa3}
\begin{cases}
\L_1(v_\ve,w_\ve):=-K_1(v_{\ve,x}+w_{\ve,y})_x-K_3(v_{\ve,y}-w_{\ve,x})_y=\d\ve^{-2}v_\ve(1-|u_\ve|^2),
\\
\L_2(v_\ve,w_\ve):=-K_1(v_{\ve,x}+w_{\ve,y})_y+K_3(v_{\ve,y}-w_{\ve,x})_x=\ve^{-2}w_\ve(1-|u_\ve|^2).
\end{cases}
\ee
\el

\noindent
Here and in what follows, we use a subscript notation for partial or directional derivatives: $w_x=\d\frac {\p w}{\p x}$, $u_\tau:=\d\frac{\p u}{\p\tau}$, etc.

\smallskip
We next note that
the second order constant coefficients linear system $\L:=(\L_1,\L_2)$ is elliptic, in the sense that it satisfies the strong Legendre-Hadamard ellipticity condition (see, e.g., \cite{giaquinta}*{Chapter I, (1.9)}). To justify this observation, we note that $\L$ arises from the energy functional $E_0(u)$. Writing (only in this paragraph) 
$$u=(u^1,u^2),\ p^i=(p^i_1,p^i_2)=\na^t u^i,\ i=1,2,$$ the energy density $e_0(u)$ may be identified with the following  function of $(p^1, p^2)$:
\bes
e_0(p^1,p^2)=\frac{K_1}2 (p^1_1+p^2_2)^2+\frac{K_3}2(p^2_1-p^1_2)^2,
\ees
and thus, for every $\xi=(\xi_1, \xi_2)$ and $\lambda=(\lambda^1,\lambda^2)$, we have
\be
\l{aa30}
\ba
\sum_{1\le i, j, \alpha,\beta\le 2}\frac{\p^2e_0}{\p p^i_\alpha \p p^j_\beta}\xi_\alpha\xi_\beta\lambda^i\lambda^j=& K_1(\xi_1\lambda^1+\xi_2\lambda^2)^2+K_3(\xi_1\lambda^2-\xi_2\lambda^1)^2
\\
\ge & K_3 [(\xi_1\lambda^1+\xi_2\lambda^2)^2+(\xi_1\lambda^2-\xi_2\lambda^1)^2]=K_3 |\xi|^2|\lambda|^2,
\ea
\ee
which shows that, indeed, $\L$ is elliptic. 

\smallskip
An alternative route to ellipticity consists of identifying $u$ with the $1$-form $\zeta=v \, dx+w\, dy$, noting that 
\bes
E_0(u)=\frac{K_1}2\int_\O |d^\ast\zeta|^2+\frac{K_3}2\int_\O |d\zeta|^2,
\ees
and then using the ellipticity of the Hodge system $\begin{cases}  d\zeta=f,
\\
d^\ast \zeta=f^\ast.\end{cases}$

\smallskip
This observation allows us to apply to $\L$ the regularity theory for elliptic systems as in \cites{dn,adn}. However, since we will rely on estimates in variable domains and with variable operators, we present here the statements instrumental for our purposes, with elements of proofs.  

\smallskip
We first quantify \textcolor{red}{the uniform ellipticity of   the} operator $\L$, by introducing the assumption
\be
\l{ak1}
0\le\delta\le \delta_1<1,
\ee
where $\delta_1$ is a fixed constant.

\smallskip
We fix a smooth bounded domain $\O$ and a boundary datum $g\in C^\infty(\p\O ; \C)$. A ball $B=B_r(x)$ is {\it admissible} if either 
$B\subset \O$, or the center of $B$ is on $\p\O$. We set 
\begin{equation}
\label{eq:bst}
B_\ast:=B_{r/2}(x).
\end{equation}
Consider a solution $u$ of 
\be
\l{ah2}
\begin{cases}
\L u=f&\text{in }B\cap \O
\\
u=g&\text{on }B\cap\p\O
\end{cases}
\ee
(the last condition being empty if $B\subset\O$). Note that, for small $r$, if the ball $B$ is centered at some $x\in\p\O$, then $B\cap\O$ is a Lipchitz open set and $B\cap\p\O$ is a Lipschitz portion of $\p (B\cap\O)$. Therefore, the second condition in \eqref{ah2} makes sense provided, say, $u\in H^1(B\cap\O)$.
In what follows, we always make the implicit assumption that $r$ is sufficiently small so that these considerations apply.
Note that this smallness assumption does not depend on $\ve$ or $\delta$.

\bl
\l{ah1}
Assume \eqref{ak1}.
Let $0<\alpha<1$ and set $q=q(\alpha):=\d\frac 2{2-\alpha}\in (1, 2)$. Let $p>2$.
Let $B=B_r(x)$ be an admissible ball and consider a solution $u\in H^1(B\cap \O)$ of \eqref{ah2}.
\ben
\item (Interior estimates)  If $B\subset\O$, then (for some absolute constants $C_{\alpha,\delta_1}$ and $C_{p, \delta_1}$)
\begin{gather}
\l{ah3}
r^\alpha\frac{|u(y)-u(z)|}{|y-z|^\alpha}\le C_{\alpha,\delta_1}(\vertii{\na u}_{L^2(B)}+r^\alpha\vertii{f}_{L^q(B)}),\ \fo y, z\in B_\ast,
\\
\l{ah31}
r^\alpha\frac{|u(y)-u(z)|}{|y-z|^\alpha}\le C_{\alpha,\delta_1}(\vertii{\na u}_{L^2(B)}+r\vertii{f}_{L^2(B)}),\ \fo y, z\in B_\ast,
\\
\l{ah4}
r|\na u(y)|\le C_{p,\delta_1}(\vertii{\na u}_{L^2(B)}+r^{2-2/p}\vertii{f}_{L^p(B)}),\ \fo y\in B_\ast.
\end{gather}
\item (Boundary estimates)
There exists some finite $r_0>0$ such that, if 
 $r\le r_0$ and $x\in\p\O$, then (for some  constants $C_{\alpha, \delta_1,\O}$ and $C_{p,\delta_1,\O}$)
\begin{gather}
\l{ah4a}
\ba
r^\alpha\frac{|u(y)-u(z)|}{|y-z|^\alpha}\le C_{\alpha,\delta_1,\O}(\vertii{\na u}_{L^2(B\cap\O)}+&r^\alpha\vertii{f}_{L^q(B\cap\O)}+r|g|_{\Lip (B\cap\p\O)}),\\ &\fo y, z\in B_\ast,
\ea
\\
\l{ah4b}
\ba
r^\alpha\frac{|u(y)-u(z)|}{|y-z|^\alpha}\le C_{\alpha,\delta_1,\O}(\vertii{\na u}_{L^2(B\cap\O)}+&r\vertii{f}_{L^2(B\cap\O)}+r|g|_{\Lip (B\cap\p\O)}),\\ &\fo y, z\in B_\ast,
\ea
\\
\l{ah5}
\ba
r|\na u(y)|\le C_{p,\delta_1,\O}(\vertii{\na u}_{L^2(B\cap\O)}+&r^{2-2/p}\vertii{f}_{L^p(B\cap\O)}+r|g|_{\Lip (B\cap\p\O)}\\
+&r^2|\p g/\p\tau|_{\Lip (B\cap\p\O)}),\
\fo y\in B_\ast.
\ea
\end{gather}
\een
\el

Note the scaling (in the radius $r$) of the estimates, which comes from the fact that we work in two dimensions and that $\L$ is a homogeneous second order system.

\bp[Idea of proof of Lemma \ref{ah1}] After scaling, item {\it 1} is a special case of the interior estimates for elliptic systems \cite{dn}, \cite{giaquinta}*{Chapter 3, Theorem 2.2}, combined with the embedding $H^2_{loc}\hookrightarrow C^\alpha$, $0<\alpha<1$. Note that here the scaling argument relies on the homogeneity of $\L$. Again after scaling and (for small $r$) flattening of the boundary, item {\it 2} follows from the model case $B\cap\O=\{ (x,y)\in B_1(0);\, y>0\}$. Some care is needed since the flattening depends on $x$ and $r$, and one has to make sure that one can choose constants independent of $x$, $r$, and $\delta_1$ in the method of freezing of the coefficients. This is indeed possible for sufficiently small $r$ (see, e.g., the detailed proofs in \cite{campanato}*{Cap. III} or \cite{gt}*{proof of Theorem 9.13}).
\ep

Iterating the proof of Lemma \ref{ah1} for our specific system \eqref{aa3} and taking $r=\ve$, we obtain the following result, that we state here without proof. 
\bl
\l{aw6}
Assume \eqref{ak1}. Fix $g\in C^\infty(\p
\O ; \so)$.  Let $u=u_\ve$, $0<\ve\le 1$, be critical points of $E_\ve$ in $H^1_g(\O)$ satisfying the {\it a priori} bound 
\be
\l{aw7}
|u(x)|\le M<\infty,\ \fo \ve,\, \fo x\in\O.
\ee

Then there exist finite constants $C_k$ depending on $M$, $\delta_1$, $\O$, and $g$ such that
\be
\l{aw8}
|D^k u(x)|\le C_k \ve^{-k},\  \fo x\in\overline\O,\, \fo k\in\N.
\ee

Moreover, with finite constants $\widetilde C_k$ depending on $M$ and $\delta_1$ (but not on $\O$ or $g$), we have
\be
\l{aw8a}
|D^k u(x)|\le \widetilde C_k \ve^{-k},\  \fo x\in\O\ \text{s.t. }\dist (x, \p\O)\ge\ve,\, \fo k\in\N.
\ee
\el

Next, we note an important consequence of the ellipticity of $\L$. The system \eqref{aa3} is of the form
\be
\l{av1}
\L u=F(u),\ \text{with }F(u)=\ve^{-2} u(1-|u|)^2.
\ee

Noting that $F$ is analytic, we have the following result, essentially established by Morrey \cite{morrey} (see also Petrowsky \cite{petrowsky}).
\bl
\l{av2}
Let $U\subset\R^2$ be an open set.
If $u\in H^1_{loc}(U)$ is a weak solution of \eqref{av1}, then $u$ is analytic.
\el
\bp
Let us note that, by standard regularity theory \cite{dn}, the 2D-embedding $H^1_{loc}\hookrightarrow L^p_{loc}$, $\fo p<\infty$, and the fact that our $F$ has polynomial growth, we have $u\in C^\infty$. We next note that the Legendre-Hadamard ellipticity condition checked in \eqref{aa30} implies the ellipticity in the sense of Douglis and Nirenberg \cite{dn}*{Section 1}. This is a general fact, but we illustrate it in our special case. For a second order 2D-variational  system with energy density $e_0(p^1, p^2)$, the ellipticity in the sense of \cite{dn} requires that the following determinant
\be
\l{av3}
D(\xi):=\det \, \left(
\sum_{1\le\alpha,\beta\le 2}\frac{\p^2e_0}{\p p^i_\alpha \p p^j_\beta}\xi_\alpha\xi_\beta
\right)_{1\le i, j\le 2}
\ee
does not vanish when $\xi=(\xi_1, \xi_2)\in\R^2\setminus\{0\}$.

\smallskip
Considering the left-hand side of \eqref{aa30} as a quadratic form in $\lambda$ with $\xi$-depending coefficients, the determinant in \eqref{av3} is nothing but the determinant of this quadratic form. Thus, by \eqref{aa30}, $D(\xi)>0$, $\fo \xi\neq 0$, as claimed. (Of course, one could check \eqref{av3} directly by noting that $D(\xi)=K_1 K_3 |\xi|^2$.) 

\smallskip
Finally, the main result in Morrey \cite{morrey} asserts that smooth solutions of analytic elliptic systems are analytic,  implying the conclusion of the lemma.
\ep

\section{\texorpdfstring{$\eta$}{eta}-ellipticity}
\l{ab3}

\textit{Throughout this section, we assume    \eqref{ak1}}.
Let
 $\O$ and the boundary datum $g\in C^\infty(\p\O ; \so)$ be fixed. Let $u_\ve$ be a minimizer of $E_\ve$ in $H^1_g(\O ; \C)$.
We will establish conditional \textit{a priori} estimates on $u_\ve$, with constants depending on $\delta_1$, but not on $\delta$ satisfying \eqref{ak1}. These constants  will possibly  depend on $\O$ or $g$ and the estimates will be valid for $\ve\le \ve_0$, with $\ve_0$ possibly depending on $\O$ and $g$. 

\smallskip

The main result of this section is 
the following.
\bl
\l{ab8}
Let  $0<\alpha<1$ and $\lambda>0$ be fixed. Then there exist absolute constants $\eta>0$ and $M<\infty$ (depending only on $\delta_1$, $\alpha$, $\lambda$)  and a constant $\ve_0>0$ depending on $g$ and $\O$ such that:
\be
\l{ab4}
\ba
{}
[0<\ve\le\ve_0,\, B_{\ve^\alpha}(x)&\text{ admissible},\, E_\ve(u_\ve, B_{\ve^\alpha}(x)\cap\O)\le\eta |\ln \ve|]
\\
&\implies [||u_\ve(x)|-1|\le\lambda,\, |\na u_\ve(x)|\le M/\ve]. 
\ea
\ee
Moreover, we may choose $M$ independent of $0<\lambda<1$.
\el

We next state some intermediate results (to be proved later) that will be needed in the proof of Lemma \ref{ab8}. The first result is well-known in the Ginzburg-Landau literature.

\bl
\label{ab6}
1. 
Let $\mu>0$ be fixed. Then there exists an absolute finite positive constant $\nu$  (depending only on $\mu$) such that:
\be
\l{ab7}
\begin{aligned}
&
\left[B_r(x)\subset\O, \, 0<\ve\le r,\, f:C_{r}(x)\to\C,
 r\int_{C_{r}(x)}|f_\tau|^2+\frac{r}{\ve^2}\int_{C_{r}(x)}(1-|f|^2)^2\le \nu\right]
\\
&
\implies \left[\bigg|\int_{C_r(x)}f\wedge f_\tau\bigg|\le\mu
\ \&\ \exists \, h\in H^1_f(B_r(x))\text{ s.t. } G_\ve(h, B_r(x))\le \mu \right]. 
\end{aligned}
\ee

2. Let $\mu>0$ be fixed. Then there exists a finite positive absolute constant $\nu$  (depending only on $\mu$) 
 and  a constant $r_0$ depending on $\mu$, $\O$, and $g$, such that:
\be
\begin{aligned}
\l{ag1}
\bigg[&x\in\p\O, \, 0<\ve\le r\le r_0,\, f\in H^1(\p (B_{r}(x)\cap\O) ; \C),\, f=g\ \text{on } B_r(x)\cap\p\O,
\\
&
r\int_{C_{r}(x)\cap\O}|f_\tau|^2+\frac{r}{\ve^2}\int_{C_{r}(x)\cap\O}(1-|f|^2)^2\le \nu\bigg]\implies 
\\
\bigg[&\bigg|\int_{C_r(x)\cap\O}f\wedge f_\tau\bigg|\le\mu\ \& 
\  \exists \, h\in H^1_f(B_r(x)\cap\O)\text{ s.t. } G_\ve(h, B_r(x)\cap\O)\le \mu \bigg]. 
\end{aligned}
\ee
\el

The proof of Lemma \ref{ab6} also leads to Lemmas \ref{ac1} and \ref{ac3}, that we note, without proof, for further use. 
\bl
\l{ac1}
Let $B=B_r(x)$. Fix some $s>0$. Then there exists a finite constant $t>0$ (depending only on $s$) and a finite constant $r_1>0$ depending on $s$, $\O$,  and $g$ such that 
\be
\l{ac2}
\ba
{}
\bigg[&r\le r_1,\, B\text { admissible},\,  v\in H^1(\p(B\cap\O) ;\C), \, v=g\text{ on }B\cap\p\O,
\\
&r\int_{\p B\cap\O}|v_\tau|^2+\frac 1r\int_{\p  B\cap\O}(1-|v|^2)^2\le t \bigg]
\implies \bigg|\int_{\p B\cap\O}v\wedge v_\tau\bigg|\le s.
\ea
\ee
\el
\bl
\l{ac3}
Let $B=B_r(x)$. Fix some $t>0$. Then there exists a finite constant $s>0$ (depending only on $t$) and a finite constant $r_1>0$ depending on $t$, $\O$,  and $g$ such that 
\be
\l{ac4}
\ba
{}
\bigg[&r\le r_1,\, B\text { admissible},\,  v\in H^1(\p(B\cap\O) ;\C), \, v=g\text{ on }B\cap\p\O,
\\
 & r\int_{\p(B\cap\O)}|v_\tau|^2+\frac 1r\int_{\p(B\cap\O)}(1-|v|^2)^2\le t\bigg]
\implies \bigg|\int_{\p (B\cap\O)}v\wedge v_\tau \bigg|\le s.
\ea
\ee
\el

Note that, in Lemma \ref{ac1} we prove existence of $t,$ given $s,$ while the opposite is shown in Lemma \ref{ac3}, where $t$ is given and existence of $s$ follows.

\smallskip
 The final auxiliary result used in the proof of Lemma \ref{ab8}  relies on Lemma \ref{ah1}.
\bl
\l{ah6}
Let $B=B_r(x)$. Let $u=u_\ve$ be a minimizer of $E_\ve$ in $H^1_g(\O ; \C)$.  Let $0<s\le 1$. Then there exists some finite constant $t>0$ (depending only on $s$) and a finite constant $r_1>0$ depending on $s$,  $\O$, and $g$  such that 
\be
\l{ah7}
\ba
\bigg[&0<4\ve\le r\le r_1,\, B\text{ admissible},\, E_\ve (u, B\cap\O)\le s,\\
 &
r\int_{\p B\cap\O}|u_\tau|^2+\frac 1r\int_{\p B\cap\O}(1-|u|^2)^2\le s\bigg]
\implies  |1-|u(z)||\le t,\, \fo z\in B_\ast.
\ea
\ee
Moreover, we may choose $t=t(s)$ such that $
\lim_{s\to 0}t(s)=0$.
\el

\noindent
(Recall that $B_\ast:=B_{r/2}(x)$.)

\smallskip
We now return to the proof of Lemma \ref{ab8}. In what follows,  $C_j$ is a generic constant independent of  $\ve$ or the center of the ball.

\bp[Proof of Lemma \ref{ab8}, using Lemmas \ref{af1}, \ref{ab6}--\ref{ah6}]
Fix some constant $\alpha_1$ such that $0<\alpha<\alpha_1<1$. We distinguish the cases $B_{\ve^{\alpha_1}}(x)\subset \O$, respectively $B_{\ve^{\alpha_1}}(x)\not\subset \O$. In what follows, $\ve$ is sufficiently small and not fixed, while $\eta$ and $s>0$ are constants that we will select at the end of the proof.

\smallskip
\noindent
{\it Case 1. $B_{\ve^{\alpha_1}}(x)\subset \O$.}
 Clearly, we have $E_\ve(u_\ve,\O)\le C_1|\ln\ve|$ and thus, by Lemma \ref{af1} applied with $\omega=\O$, we have 
 \be
 \l{ai1}
 G_\ve(u_\ve,\O)\le C_2|\ln\ve|. 
 \ee
Fix $\alpha_1<\beta<\gamma<1$. By \eqref{ai1} and the mean value theorem, there exists some $\ve^\beta<r_1<\ve^{\alpha_1}$ such that
 \be
 \l{ad3}
 r_1\int_{C_{r_1}(x)}|u_{\ve,\tau}|^2+\frac{r_1}{\ve^2}\int_{C_{r_1}(x)}(1-|u_\ve|^2)^2\le C_3.
 \ee
 By Lemmas \ref{af1} and \ref{ac3}, this implies, for sufficiently small $\ve$, 
 \be
 \l{ad4}
 \ba
 G_\ve(u_\ve, B_{r_1}(x))\le &\frac 1{1-\delta_1}E_\ve(u_\ve, B_{r_1}(x))+C_4 \le \frac 1{1-\delta_1}E_\ve(u_\ve, B_{\ve^\alpha}(x))+C_4
 \\
 \le & C_5\eta |\ln\ve|.
 \ea
 \ee
Note the important fact that, while $C_1, C_2, C_3, C_4$ depend on $g$, $C_5$ and the constant $C_6$ below are universal constants, depending only on $\delta_1$, $\alpha$, $\alpha_1$, $\beta$, $\gamma$.
By \eqref{ad4} and the mean value theorem, there exists some $\ve^\gamma<r_2<\ve^\beta$ such that 
 \be
 \l{ad5}
 r_2\int_{C_{r_2}(x)}|u_{\ve,\tau}|^2+\frac{r_2}{\ve^2}\int_{C_{r_2}(x)}(1-|u_\ve|^2)^2\le C_6\eta.
 \ee
 
 By \eqref{ad5}, and Lemmas \ref{ac1} and \ref{ab6}, for sufficiently small $\eta$ (depending on $s$) we have (with $h$ the competitor given by Lemma \ref{ab6})
 \be
 \l{ae1}
 E_\ve(u_\ve, B_{r_2}(x))\le E_\ve(h,B_{r_2}(x))\le C_7 s. 
 \ee

 The first conclusion in \eqref{ab4} follows from \eqref{ad5}, 
 \eqref{ae1} (with sufficiently small $s$), and Lemma \ref{ah6}.
 
 \smallskip
 The second part of \eqref{ab4} follows from the first part of \eqref{ab4} and estimate \eqref{ah4} in Lemma \ref{ah1} item {\it 1} (applied with $r=\ve$).
 
 \smallskip
 \noindent
 {\it Case 2.  $B_{\ve^{\alpha_1}}(x)\not\subset \O$.} The idea is similar, but this time we rely on estimate \eqref{ah5} in  Lemma \ref{ab6} item {\it 2}. Let $\alpha<\alpha_2<\alpha_1$. Let $y$ be the nearest point projection of $x$ on $\p\O$. Clearly, for small $\ve$, the admissible ball $B_{\ve^{\alpha_2}}(y)$ is contained in $B_{\ve^\alpha}(x)$ and contains $B_\ve(x)$. We proceed as in the proof of \eqref{ae1} and find that $E_\ve(u_\ve, B_{\ve^{\alpha_2}}(y)\cap\O)\le C_7s$, which is the   analogue of \eqref{ae1} adapted to Case 2.  We conclude as above.  
 \ep
 
 We now proceed to the proofs of the auxiliary results.
 \bp[Proof of Lemma \ref{ab6} item 1]
Set $\D:=B_1(0)$. By scaling, we have to prove the following, for a sufficiently small $\nu$, and with $t:=\ve/r\le 1$:
 \be
 \l{af8}
 \ba
 &{}
\left[0<t\le 1,\, f:\so\to\C,
\int_{\so}|f_\tau|^2+\frac{1}{t^2}\int_{\so}(1-|f|^2)^2\le \nu\right]
\\
&\implies \left[\bigg|\int_{\so}f\wedge f_\tau\bigg|\le\mu
\ \& \ \exists \, h\in H^1_f(\D ; \C)\text{ s.t. } G_{t}(h, \D)\le \mu \right]. 
\ea
 \ee
 
We first note that
 \be
 \l{af9}
 \ba
 |f|^2=&(|f|^2-1)+1\le \frac 12(1-|f|^2)^2+\frac 12 +1\le \frac{1}{2t^2}(1-|f|^2)^2+\frac 32.
 \ea
 \ee
 Combining \eqref{af9} with Cauchy-Schwarz, we find that
 \bes
 \bigg|\int_{\so}f\wedge f_\tau\bigg|^2\le \left(\frac 1{2 t^2}\int_{\so}(1-|f|^2)^2+3\pi\right)\int_{\so}|f_\tau|^2\le (\nu/2+3\pi)\nu,
 \ees
 whence the first part of \eqref{af8} if $(\nu/2+3\pi)\nu\le\mu^2$.
 
 \smallskip
 Concerning the second part of \eqref{af8}, we first note that, for small $\nu$ independent of $t\le 1$, under the assumption of \eqref{af8} we have 
 \be
 \l{af56}
 1/2\le |f|\le 3/2.
 \ee
 
  A cheap way to establish this fact consists of noting that, if a sequence satisfies
 \bes
 \int_{\so}|f_{j,\tau}|^2+\int_{\so}(1-|f_j|^2)^2\to 0\ \text{as }j\to\infty,
 \ees
 then $|f_j|\to 1$ uniformly as $j\to\infty$. 
Alternatively, one may  use the inequality 
\bes
|f(e^{\im \theta})-f(e^{\im\va})|^2\le |\theta-\va|\int_{\so}|f_\tau|^2,\ \fo \theta-\pi\le\va\le\theta+\pi,
\ees
and check that
 \eqref{af56} holds, e.g., when  $\nu\le 7\sqrt 2/64$.

 \smallskip
 Consider $\nu$ such that \eqref{af56} holds for every $f$ satisfying the assumption of \eqref{af8}. Writing, locally, $f=\rho e^{\im\psi}$, we have
 \bes
 \frac 14 \int_{\so}|\psi_\tau|^2\le \int_{\so}|f_\tau|^2\le\nu,
 \ees 
 and thus 
 \bes
 \int_{\so}|\psi_\tau|<2\pi,
 \ees
 provided $\nu<\pi/2$. Therefore, for small $\nu$, $f/|f|$ has zero degree and $\psi$ is globally defined.  
  
  \smallskip
 We now define our competitor 
 \begin{gather*}
h(re^{\im\theta}):=F(re^{\im\theta})e^{\im L(r e^{\im\theta})},\ 0\le r\le 1,\, \theta \in\R,
 \\
 \intertext{where}
 F(re^{\im\theta}):=(1-r)+r \rho(e^{\im\theta})=(1-r)+r |f|(e^{\im\theta}),
 \\
 L(r e^{\im\theta}):=(1-r)a+r \psi(e^{\im\theta}),\ \text{with }a:=\fint \psi.
 \end{gather*}
 
 Clearly, thanks to \eqref{af56}, we have
 \begin{gather}
 \l{ag3}
 (1-|F(re^{\im\theta})|^2)^2\le (1-|f(e^{\im\theta})|^2)^2,
 \\
 \l{ag4}
 |\na F(re^{\im\theta})|^2=(1-|f(e^{\im\theta})|)^2+\left|\frac{d |f(e^{\im\theta})|}{d \theta}\right|^2\le (1-|f(e^{\im\theta})|^2)^2+\left|\frac{d f(e^{\im\theta})}{d \theta}\right|^2,
 \\
 \notag
 |\na L(re^{\im\theta})|^2=(a-\psi(e^{\im\theta}))^2+\left|\frac{d \psi(e^{\im\theta})}{d\theta}\right|^2,
 \end{gather}
 and thus, using the definition of $a$ and Poincar\'e's inequality,
 \be
 \l{ag5}
 \int_\D|\na L|^2\le 2 \int_{\so}\left|\frac{d \psi}{d\theta}\right|^2.
 \ee
 
For small $\nu$ (depending on $\mu$), the second part of \eqref{af8} follows from  the estimates \eqref{ag3}--\eqref{ag5}.
 \ep
 
 \bp[Proof of Lemma \ref{ab6} item 2] The first part of \eqref{ag1} is proved exactly as the first part of \eqref{ab7}.
 
 \smallskip
 We will reduce  the second part of \eqref{ag1} to the situation considered in item 1. Let $r_0$ be sufficiently small (depending on $\O$) and $C_0$ be a sufficiently large universal constant such that, for  $x\in \p\O$ and $0<r\le r_0$, there exists a bi-Lipschitz homeomorphism $\Phi=\Phi_{x,r}:\overline B_r(x)\cap \overline\O\to \overline B_r(0)$ such that $\|D\Phi\|_\infty\le C_0$, $\|D\Phi^{-1}\|_\infty\le C_0$, and $\Phi (\overline B_r(x)\cap\p\O)=\{ x+\im y\in C_r (0);\, y\le 0\}$.
After composing with $\Phi^{-1}$ and using scale invariance, the second part of \eqref{ag1} amounts to proving \eqref{ag2} below.
Set  $\so_+:=\{ x+\im y\in\D;\, y\ge 0\}$, and define similarly $\so_-$. Then, for a sufficiently small $\nu_1$ (depending only on $\mu$) and a sufficiently small $r_1$ (depending on $\mu$ and on a fixed given constant $M$), we have
 \be
 \l{ag2}
 \ba
 \bigg[&0<t\le 1, f\in H^1(\so ; \C), \, |f|=1\text{ and }|f_\tau|\le Mr_1\text{ on }\so_-,
 \\
 &\int_{\so_+}|f_\tau|^2+\frac 1{t^2}\int_{\so_+}(1-|f|^2)^2\le \nu_1\bigg]
\implies \exists\, h\in H^1_f(\D; \C)\text{ s.t. }G_t(h, \D)\le \mu.
 \ea
 \ee
 (In our case, the constant $M$ itself depends only on  $C_0$ and on the Lipschitz constant of $g$.)
 
 \smallskip
 In order to prove the existence of $\nu_1$ and $r_1$ (and thus to complete the proof of the lemma), we note that, if $\nu$ is as in item {\it 1}, then \eqref{ag2} holds provided $\nu_1+\pi (Mr_1)^2<\nu$. It then suffices to let $\nu_1<\nu/2$ and $r_1<\sqrt{2\pi\nu}/M$. 
 \ep
 
 \bp[Proof of Lemma \ref{ah6}] We consider only the case where $B\subset\O$. As explained in the proof of Lemma \ref{ab6} item {\it 2}, the other case is similar.
 
 \smallskip
By estimate \eqref{af4} in Lemma \ref{af1} and Lemma \ref{ac3}, there exists some finite constant $C_1>0$ independent of $s\le 1$ such that, if the assumptions of  \eqref{ah7} hold for such $s$, then 
\be
\l{aj2}
G_\ve(B)=\frac 12\int_B |\na u|^2+\frac 1{4\ve^2}\int_B (1-|u|^2)^2\le C_1 s\le C_1.
\ee

We next note that, for some appropriate constant $C_2$, we have
\be
\l{aj3}
|w(1-|w|^2)|^{4/3}\le C_2((1-|w|^2)^2+1),\ \fo w\in\C.
\ee

Let $B'$ be a ball of size $2\ve$ contained in $B$.
Applying \eqref{aj3} with $z=u(x)$, integrating over $B'$,  and using \eqref{aj2}, we find that
\be
\l{aj4}
\vertii{\ve^{-2}u(1-|u|^2)}_{L^{4/3}(B')}^{4/3}\le C_2\ve^{-8/3}\int_{B'}((1-|u|^2)^2+1)\le C_3\ve^{-2/3}.
\ee

Combining estimate \eqref{ah3} (applied, in $B'$, with $\alpha=1/2$ and thus $q=4/3$), \eqref{aj2}, and \eqref{aj4}, we find that
\be
\l{aj5}
\ve^{1/2}\frac{|u(y)-u(z)|}{|y-z|^{1/2}}\le C_4,\ \fo y, z\in (B')_\ast,
\ee
and thus, in particular, 
\be
\l{aj6}
|u(y)-u(z)|\le C_5, \ \fo y, z\in (B')_\ast,
\ee
where $C_5$ is independent of $s\le 1$.

\smallskip
Combining now \eqref{aj6} with \eqref{aj2}, we find that
\be
\l{aj7}
|u(y)|\le C_6,\ \fo y\in (B')_\ast,
\ee
again with $C_6$ independent of $s\le 1$.

\smallskip
We next note that, for small $w$,  \eqref{aj3}  can be improved as follows :
\be
\l{aj8}
|w|\le C_6\implies |w(1-|w|^2)|^2\le (C_6)^2 (1-|w|^2)^2.
\ee

Arguing as above and using the first inequality in \eqref{aj2}, \eqref{aj8} (instead of \eqref{aj3}), \eqref{aj7}, and \eqref{ah31} (instead of \eqref{ah3}), we find that 
\be
\l{ak1a}
|u(y)-u(z)|\le C_7\sqrt{s},\, \fo y, z\in (B')_\ast,
\ee
with $C_7$ independent of $0<s\le 1$. 

\smallskip
Finally, \eqref{ak1a} and \eqref{aj2} imply \eqref{ah7}, with $t(s)\to 0$ as $s\to 0$.
\ep

\section{Pohozaev type identities and a priori estimates}
\l{ak2}

In this section, we  derive  the Pohozaev identity corresponding to the operator $\L$ in \eqref{aa3}. As for the Dirichlet integral, the identity is obtained by multiplying \eqref{aa3} with $(x-x^0)u_x+(y-y^0)u_y$. For simplicity, we perform our calculations when $x^0=y^0=0$, but in subsequent results we may take other values of $x^0$ and $y^0$. Remarkably, the Pohozaev identity implies {\it a priori} estimates merely under the $\delta$-independent assumption that $\O$ is star-shaped. The idea of using the Pohozaev identity is natural in this context. For the standard Ginzburg-Landau equation, it was successfully used in \cite{bbh} and subsequently \cite{struwe}, \cite{sandier_serfaty}. For our specific system and in a disc, it appears in Kowalczyk, Lamy, and  Smyrnelis \cite{kowalczyk}*{Section 5}.

\bl
\l{ak3}
(General Pohozaev identity) Let $\omega$ be a Lipschitz bounded domain.  Let $X=(x,y)$ denote the 'generic' point in $\R^2$. Let  $\nu$, respectively $\tau$, denote the unit outward normal, respectively the unit directly oriented tangent vector to $\p\omega$. 

Set
\bes
V:=xv_x+yv_y,\, W:=xw_x+yw_y,\, Z=(V,W)\sim x u_x+yu_y.
\ees

Let $u\in C^3(\overline\omega ; \C)$ be a  critical point of $E_\ve$. 
Then 
\be
\l{am3}
\begin{aligned}
\frac 1{\ve^2}\int_\omega (1-|u|^2)^2=&\frac 1{2\ve^2}\int_{\p\omega} (1-|u|^2)^2 (X\cdot\nu)
\\
&
-2K_1\int_{\p\omega} (\div u)(Z\cdot\nu)
-2K_3\int_{\p\omega} (\curl u) (Z\cdot\tau)
\\
&
+K_1\int_{\p\omega} (\div u)^2 (X\cdot\nu)+K_3\int_{\p\omega} (\curl u)^2 (X\cdot\nu).
\end{aligned}
\ee
\el

\bp
We mimic the proof of Pohozaev's identity. 
We rewrite \eqref{aa3} as
\be
\l{am0}
\begin{cases}
-K_1 (\div u)_x+K_3(\curl u)_y=\ve^{-2}v(1-|v|^2)
\\
-K_1(\div u)_y-K_3(\curl u)_x=\ve^{-2}w(1-|u|^2)
\end{cases}.
\ee

We let $B_j$ denote a boundary term that we will make explicit at the end of the proof.

\smallskip
Multiplying the first equation in \eqref{am0} with $V$ and the second one with $W$, integrating once by parts and summing up the results, we find that
\be
\l{am1}
-B_1+K_1\int_\omega (\div u)(\div Z)+K_3\int_\omega (\curl u)(\curl Z)=-B_2+\frac 1{2\ve^2}\int_\omega (1-|u|^2)^2.
\ee

We next note the 2D-identities 
\be
\l{am2}
(\div u)(\div Z)=\frac 12 \div [(\div u)^2 X],\ (\curl u)(\curl Z)=\frac 12 \div[(\curl u)^2 X].
\ee

Inserting \eqref{am2} into \eqref{am1} and integrating, we find that
\bes
-B_1+K_1 B_3+K_3 B_4=-B_2 +\frac 1{2\ve^2}\int_\omega (1-|u|^2)^2.
\ees

We obtain the conclusion of the lemma by noting that
\begin{gather*}
B_1=K_1\int_{\p\omega} (\div u)(Z\cdot\nu)+K_3\int_{\p\omega} (\curl u) (Z\cdot\tau), 
\\
B_2=\frac 1{4\ve^2}\int_{\p\omega} (1-|u|^2)^2 (X\cdot\nu),
\\
B_3=\frac 12\int_{\p\omega}(\div u)^2 (X\cdot \nu),
\\
B_4=\frac 12 \int_{\p\omega} (\curl u)^2 (X\cdot\nu).
\qedhere
\end{gather*}
\ep

We next rewrite the identity \eqref{am3} in normal and tangential coordinates on $\p\omega$.  

We note the following identities, with $(\mathbf{i}, \mathbf{j})$ the canonical basis of $\R^2$:
\be
\l{am3a}
\ba
\div u=&(\na v)\cdot\mathbf{i}+(\na w)\cdot \mathbf{j}=(v_\tau \tau+v_\nu\nu)\cdot\mathbf{i}+(w_\tau \tau+w_\nu\nu)\cdot\mathbf{j}\\
=&u_\tau\cdot\tau+u_\nu\cdot\nu.
\ea
\ee

We write $\nu=(\nu_x, \nu_y)$ and $\tau=(\tau_x,\tau_y)$. Using \eqref{am3a} and the identities
\bes
\nu_x=\tau_y,\ \nu_y=-\tau_x,\ \curl u=\div\, (w, -v),
\ees
we find that
\be
\l{am4}
\curl u=u_\nu\cdot\tau-u_\tau\cdot\nu.
\ee

Similarly, we have
\begin{gather}
\l{am5}
Z=\left(X\cdot\tau\right) u_\tau+\left(X\cdot\nu \right)u_\nu,
\\
\l{am6}
Z\cdot\nu=\left(X\cdot\tau\right)  \left(u_\tau\cdot\nu\right)+\left(X\cdot\nu \right)\left(u_\nu\cdot\nu\right),
\\
\l{am7}
Z\cdot\tau=\left(X\cdot\tau \right) \left(u_\tau\cdot\tau\right)+\left(X\cdot\nu \right)\left(u_\nu\cdot\tau\right).
\end{gather}

Inserting \eqref{am3a}--\eqref{am7} into \eqref{am3} and rearranging the terms,  we  obtain the following consequence of \eqref{am3}.
\bl
\l{an1}
With the notation in Lemma \ref{ak3}, we have, for any  $X^0\in\R^2$,
\be
\l{an2}
\begin{aligned}
\frac 1{\ve^2}\int_\omega (1-|u|^2)^2=&\frac 1{2\ve^2}\int_{\p\omega} (1-|u|^2)^2 ((X-X^0)\cdot\nu)
\\
&
+\int_{\p\omega} Q_1(X-X^0, u_\tau\cdot \tau, u_\tau\cdot\nu)
\\
&-\int_{\p\omega} Q_2(X-X^0, u_\nu\cdot \tau, u_\nu\cdot\nu)
\\
&+\int_{\p\omega} Q_3(X-X^0, u_\tau\cdot \tau, u_\tau\cdot\nu,u_\nu\cdot \tau, u_\nu\cdot\nu),
\end{aligned}
\ee
where the $Q_j$'s are quadratic forms  with coefficients depending on $X-X^0$,  explicitly given by
\begin{gather}
\l{an3}
\ba
Q_1(X-X^0, \xi_1, \xi_2)=&K_1((X-X^0)\cdot\nu) (\xi_1)^2+K_3((X-X^0)\cdot \nu)(\xi_2)^2
\\
&-2(K_1-K_3) ((X-X^0)\cdot\tau)\xi_1\xi_2,
\ea
\\
\l{an4}
Q_2(X-X^0, \eta_1, \eta_2)=K_3((X-X^0)\cdot \nu)(\eta_1)^2+K_1((X-X^0)\cdot\nu) (\eta_2)^2,
\\
\l{an5}
\ba
Q_3(X-X^0, \xi_1,\xi_2,\eta_1, \eta_2)=&-2K_3((X-X^0)\cdot\tau)\xi_1\eta_1
\\
&-2 K_1 ((X-X^0)\cdot\tau)\xi_2\eta_2.
\ea
\end{gather}
\el

Specializing to the case where $\omega$ is a disc, respectively a half-disc, we obtain the following consequences of our calculations. 
\bl
\l{an6}
Assume \eqref{ak1}.
Let $u\in C^3(\overline B)$ be a critical point of $E_\ve$ in a disc $B$ of radius $r$.
Then 
\bes
\frac 1{\ve^2}\int_B (1-|u|^2)^2+(1-\delta_1)r\int_{\p B} |u_\nu|^2\le \frac r{2\ve^2}\int_{\p B} (1-|u|^2)^2+(1+\delta_1)r\int_{\p B}|u_\tau|^2.
\ees
\el
\bl
\l{an8}
Assume \eqref{ak1}.  
Then there exist some finite positive constants $C_j=C_j(\delta_1)$, $j=1, 2, 3$, such that, if $u\in C^3(\overline H)$ is a critical point of $E_\ve$ in a half-disc $H$ of radius $r$, then 
\bes
\frac 1{\ve^2}\int_H (1-|u|^2)^2+C_1 r\int_{\p H} |u_\nu|^2\le \frac {C_2r}{\ve^2}\int_{\p H} (1-|u|^2)^2+C_3 r\int_{\p H}|u_\tau|^2.
\ees
\el

\bp[Proof of Lemma \ref{an6}]
The conclusion follows from \eqref{an2}--\eqref{an5} (with $X^0$ the center of $B$), combined with the observation that, in the case of a disc of radius $r$, we have $Q_3=0$ and
\begin{gather*}
Q_1(X-X^0,\xi_1, \xi_2)=K_1 r(\xi_1)^2+K_3 r(\xi_2)^2\le (1+\delta_1) r\verti{\xi}^2,
\\
Q_2(X-X^0,\eta_1, \eta_2)=K_1 r(\eta_2)^2+K_3 r(\eta_1)^2\ge  (1-\delta_1) r\verti{\eta}^2.
\qedhere
\end{gather*}
\ep

\bp[Proof of Lemma \ref{an8}]
In what follows, $C_j$ denotes a generic positive constant depending possibly on $\delta_1$.

\smallskip
With no loss of generality, we may assume that $r=1$ and
\bes
H=\{ X=(x,y)\in\R^2; |X|<1,\, y>0\}. 
\ees

Let $0<a<1$  be any fixed number, and set $X^0=(0,a)$. It is easy to see that 
\begin{gather}
\l{ap2}
(X-X^0)\cdot\nu\ge C_3>0,\ \fo X\in\p H,
\\
\l{ap4}
|(X-X^0)\cdot\nu|\le C_4, \ \fo X\in\p H,
\\
\l{ap3}
|(X-X^0)\cdot \tau|\le C_5,\ \fo X\in \p H.
\end{gather}

Combining \eqref{an2}--\eqref{an5} and \eqref{ap2}--\eqref{ap3}, we find that
\bes
\ba
\frac 1{\ve^2}&\int_H (1-|u|^2)^2+C_3(1-\delta_1)\int_{\p H}|u_\nu|^2
\\
 \le& \frac 1{\ve^2}\int_H (1-|u|^2)^2+\int_{\p H} Q_2(X-X^0, u_\nu\cdot\tau, u_\nu\cdot\nu)
\\
 \le & \frac {C_4}{2\ve^2}\int_{\p H} (1-|u|^2)^2
 +2(C_4+C_5)\int_{\p H}|u_\tau|^2
 \\
 &+4 C_4\int_{\p H}(|u_\tau\cdot\tau||u_\nu\cdot\tau|+|u_\tau\cdot\nu||u_\nu\cdot\nu|)
 \\
 \le 
 & \frac {C_4}{2\ve^2}\int_{\p H} (1-|u|^2)^2
 +2(C_4+C_5)\int_{\p H}|u_\tau|^2\\
 &+4 C_4\int_{\p H}|u_\tau||u_\nu|
  \\
 \le 
 & \frac {C_4}{2\ve^2}\int_{\p H} (1-|u|^2)^2
 +2(C_4+C_5)\int_{\p H}|u_\tau|^2\\
 &+ \frac 12 C_3(1-\delta_1)\int_{\p H}|u_\nu|^2+C_6\int_{\p H} |u_\tau|^2,
\ea
\ees
whence the conclusion of the lemma. 
\ep

By a straightforward modification of the proof of Lemma \ref{an8}, we obtain the following result in a fixed bounded domain $\O$, that we state without proof. 
\bl
\l{ap1}
Assume \eqref{ak1}. Then there exist some finite positive constants $C_j=C_j(\delta_1)$, $j=1, 2, 3$,  and $r_0=r_0(\delta_1, \O)$ such that, if $u\in C^3(\overline{B_r(x_0)\cap\O})$ is a critical point of $E_\ve$ in  $B_r(x_0)\cap\O$,  with $r\le r_0$ and $x_0\in \p\O$, then 
\be
\l{as4}
\ba
\frac 1{\ve^2}\int_{B_r(x_0)\cap\O} (1-|u|^2)^2+C_1 r\int_{\p (B_r(x_0)\cap\O)} |u_\nu|^2\le &\frac {C_2r}{\ve^2}\int_{\p (B_r(x_0)\cap\O)} (1-|u|^2)^2
\\
&+C_3 r\int_{\p (B_r(x_0)\cap\O)}|u_\tau|^2.
\ea
\ee
\el

When $\O$ is strictly star-shaped, the local estimate \eqref{as4} upgrades to a global estimate.
\begin{lemma}
\l{as5}
Assume \eqref{ak1}. Assume that $\O$ is strictly star-shaped,  i.e., that there exist some $X^0\in\O$ and $C_3>0$ such that
\be
\l{as1}
(X-X^0)\cdot \nu\ge C_3,\ \fo X\in\p\O.
\ee

Let $u=u_\ve$ be a critical point of $E_\ve$ in $H^1_g(\O)$. Then there exists a finite positive constant $C=C(\delta_1, C_3)$  such that 
\be
\l{as6}
\frac 1{\ve^2}\int_{\O} (1-|u|^2)^2+\int_{\p \O} |u_\nu|^2\le C \int_{\p \O}|g_\tau|^2.
\ee
\end{lemma}
\begin{proof}
It suffices to copy the proof of Lemma \ref{an8}. There, the existence of $C_3$ follows from the geometry of $H$. In our case, the existence of $C_3$ is an assumption.
\end{proof}

Lemmas \ref{an6} and \ref{ap1} yield the following {\it a priori} estimates for critical points of $E_\ve$ satisfying a natural bound on the energy. In particular, thanks to the energy estimate \eqref{aq2}, these bounds apply to minimizers of $E_\ve$ in $H^1_g(\O ; \C)$.
Note, however, that the estimates below do not imply \eqref{ab4}, since the constants we obtain below depend on the energy bound (which in turn depends on the boundary datum $g$). 
\bl
\l{aq3}
Assume \eqref{ak1}. 
Let $u=u_\ve$ be  critical points of $E_\ve$ in $H^1_g(\O)$ satisfying the bound
\be
\l{aq4}
E_\ve(u)\le K |\ln\ve|.
\ee

Then, with finite constants $C_j=C_j(K, \delta_1)$, $j=1, 2$,  we have, for small $\ve$,
\begin{gather}
\l{aqa}
|u|\le C_1,
\\
\l{aqb}
|\na u|\le C_2/\ve.
\end{gather}

In particular, if $u$ minimizes $E_\ve$ in $H^1_g(\O)$, then $C_1$, $C_2$ may be chosen to depend only on $\deg g$ and $\delta_1$.
\el

Combining Lemma \ref{aq3} with Lemma \ref{aw6}, we obtain the following
\begin{coro}
\l{aw9}
Assume \eqref{ak1}. Fix $g\in C^\infty(\p
\O ; \so)$.  Let $u=u_\ve$, $0<\ve\le 1$, be critical points of $E_\ve$ in $H^1_g(\O)$ satisfying the  energy bound \eqref{aq4}. Then \eqref{aw8} holds. 
\end{coro}

\bp[Proof of Lemma \ref{aq3}]
The proof is similar to the one of Lemma \ref{ah6}, with the variation that, for clarity, we perform a blow-up.

\smallskip
In what follows, $C_j$, $j\ge 3$, denotes a finite positive constant depending only on $K$ and possibly $\delta_1$, and $D_j$ a finite positive  universal constant. 
Let $x\in\overline\O$. 
With the notation in the proof of Lemma \ref{ab7}, we have either $B_{\ve^{\alpha_1}}(x)\subset\O$, or $B_{\ve^{\alpha_1}}(x)\not\subset\O$. We consider only the first case, the other one being similar. Pick some $4\ve\le \ve^{\beta}<r<\ve^{\alpha_1}$ such that
\be
\l{aq5}
r\int_{C_r(x)}|\na u|^2+\frac r{\ve^2}\int_{C_r(x)}(1-|u|^2)^2\le C_3.
\ee

By \eqref{aq5} and  Lemma \ref{an6}, we have
\be
\l{aq6}
\frac 1{\ve^2}\int_{B_r(x)}(1-|u|^2)^2\le C_4.
\ee

Assume, for simplicity, that $x=0$. Set $\widetilde u(y):=u(\ve y)$, $y\in B:=B_{4}(0)$. Then
\bes
\L \widetilde u=\widetilde f:=\widetilde u (1-|\widetilde u|^2)\ \text{in }B, 
\ees
so that, by standard elliptic estimates, 
\be
\l{aq7}
\vertii{\widetilde u}_{C^{1/2}(B_\ast)}\le D_1 \|\widetilde f\|_{L^{4/3}(B)}+D_2 \vertii{\widetilde u}_{L^4(B)},
\ee
where $B_\ast$ is as defined in \eqref{eq:bst}. By \eqref{aq6}, \eqref{aq7}, \eqref{aj3}, and the inequality
\bes
|w|^4\le D_3 (1-|w|^2)^2+1,\ \fo w\in \C,
\ees
we find (going back to $u$) that
\be
\l{aq8}
|u(z)-u(t)|\le C_5,\ \fo z, t\in B_{2\ve}(x).
\ee

Combining \eqref{aq8} with \eqref{aq6}, we find that \eqref{aqa} holds in $B_{2\ve}(x)$. Then, using \eqref{aqa} in $B_{2\ve}(x)$ and the estimate
\bes
\vertii{\na \widetilde u}_{L^\infty(B_{\ast\ast})}\le D_4 \|\widetilde f\|_{L^2(B_\ast)}+D_4\vertii{\widetilde u}_{L^\infty(B_\ast)},
\ees
(with $B_\ast=B_2(0)$, $B_{\ast\ast}=B_1(0)$)
and going back to $u$, 
we find that \eqref{aqb} holds in $B_\ve(x)$.
\ep

\begin{lemma}
\l{ar1}
Assume \eqref{ak1}.
Assume that $\O$ is strictly star-shaped. 
Consider critical points $u=u_\ve$ of $E_\ve$ in $H^1_g(\O)$ satisfying the energy bound \eqref{aq4}. Then, for some finite constant $C_1=C_1(K, \delta_1, \O, g)$ and small $\ve$, we have
\be
\l{ar2}
\int_\O |1-|u|^2|\times |\na u|^2\le C_1.
\ee
\end{lemma}
\begin{proof}
By the Gagliardo-Nirenberg inequality, \eqref{aqa}, standard elliptic estimates, and Lemma \ref{as5}, we have  the following (global in $\O$) estimates, with constants depending only on $K$, $\delta_1$, $\O$, $g$:
\be
\l{as2}
\ba
\vertii{\na u}_4^2\le & C_2 \vertii{u}_\infty\vertii{u}_{H^2}\le C_3 \vertii{u}_{H^2}\le C_4 (\vertii{g}_{H^{3/2}}+\vertii{Lu}_2)
\\
\le & C_5 (1+\ve^{-2}\vertii{1-|u|^2}_2)\le \frac{C_6}{\ve}.
\ea
\ee

We obtain \eqref{ar2} from \eqref{as6}, \eqref{as2}, and Cauchy-Schwarz.
\end{proof}

\section{Bad discs structure}
\l{as9}

In this section, we provide some easy consequences of the {\it a priori} estimates established in the previous sections. We first define the notion of bad disc. A bad disc $B=B_{C\ve}(x_\ve)$ is a disc of radius $C\ve$, centered at $x_\ve$, such that  $|u_\ve(x_\ve)|\le 1/2$ and $|u_\ve(x)|\ge 1/2$ on $\p(B\cap\O)$.   Note that in our situation we have $u_\ve\in H^1_g(\O)$, and thus the latter condition is equivalent to $|u_\ve(x)|\ge 1/2$ on $\p B\cap\O$. Here, the constant $C$ could possibly depend on a sequence $\ve_\ell\to 0$, but its size is controlled by the {\it a priori} estimates available on $u$.

\begin{lemma}
\l{at1}
Assume that \eqref{ak1} holds. We have:

\smallskip
\noindent
1. 
Suppose that $\O$ is strictly star-shaped. 
Consider critical points $u=u_\ve$ of $E_\ve$ in $H^1_g(\O)$ satisfying the energy bound \eqref{aq4}, where $\ve=\ve_\ell\to 0$. Set $A_\ve:=\{ x\in\O;\, |u(x)|\le 1/2\}$. Then there exist finite constants $N=N(K, \delta_1, \O, \vertii{g_\tau}_2)$ and $L=L(K, \delta_1)$ such that,  possibly along a subsequence $(\ve_{\ell_m})$, 
\begin{gather}
\l{at290}
\text{$A_\ve$ can be covered with at most $N$ bad discs $B_{C\ve}(x_\ve^j)$,}
\\
\intertext{for some constant $C$ possibly depending on $(\ve_\ell)$ such that}
\l{at209}
3\le C\le L,
\\
\intertext{and}
\l{at211}
\text{for $j\neq k$,  $|x_\ve^j-x_\ve^k|\ge 4 C\ve$.}
\end{gather}

Moreover, there exists some finite number $C'\ge C$, possibly depending on $(\ve_\ell)$, such that,  possibly along a  further subsequence $(\ve_{\ell_{m_n}})$, 
\be
\l{at234}
\text{$A_\ve$ can be covered with at most $N$ \enquote{enlarged} bad discs $B_{C'\ve}(x_\ve^j)$}
\ee
such that
\be
\l{at9b}
\text{for $j\neq k$,  $|x_\ve^j-x_\ve^k|\gg\ve$ as $\ve\to 0$.}
\ee
2.
The same conclusions hold if $\O$ is arbitrary and $u$ is a minimizer of $E_\ve$ in $H^1_g(\O)$, where this time $N=N(\deg g, \delta_1)$ and $L=L(\deg g, \delta_1)$.

\smallskip
\noindent
3. For each $\ve$, consider: (i) either critical points of $E_\ve$ in a strictly star-shaped domain $\O$, satisfying the energy bound \eqref{aq4} ; (ii) or minimizers of $E_\ve$. Then there exists a $C=C(\ve)$ satisfying  \eqref{at290}--\eqref{at209}.
\end{lemma}

\noindent
Note that the price to pay in order to have \eqref{at9b} is the lack of control on the constant $C'$.
\begin{proof}
1. 
By \eqref{aqb} in Lemma \ref{aq3}, there exists some $\lambda>0$ such that, if $\ve>0$ is small, we have
\be
\l{at4}
\left[x\in\O, \d\int_{B_\ve(x)\cap\O}(1-|u|^2)^2\le\lambda\right]\implies [|u|\ge 1/2\text{ in }B_\ve(x)].
\ee

Combining this fact with the {\it a priori} estimate \eqref{as6}, we find that any disjoint family of discs $B=B_{\ve/3}(x)$ such that $|u(x)|\le 1/2$ has at most $N$ elements, where $N=N(K, \delta_1, \O, \vertii{g_\tau}_2)$. Therefore, the  set $A_\ve$ can be covered with at most $N$  discs $B_i^\ve=B_\ve(x_i^\ve)$, satisfying $|u(x_i^\ve)|\le 1/2$. 

\smallskip
We next enlarge these discs in order to obtain bad discs. For this purpose, let us note the following. Fix an integer $M$ and consider, for each $\ve$, at most $M$ intervals $I_1^\ve,\ldots, I_k^\ve$, each of length $\le 2\ve$. Then, up to a sequence $\ve_\ell\to 0$, there exists some $3\le  C\le 3(M+1)$ such that $C\ve\not\in \cup_k I_k^\ve$. (A similar conclusion can be drawn if we start from a sequence $\ve_\ell\to 0$, possibly after passing to a subsequence.) Indeed, the union of the $I_k^\ve$'s cannot contain all the points $n\ve$, with $n=3, 6,\ldots, 3(M+1)$, and thus, up to a subsequence, we may take $C=3n$, for one of these   $n$'s. Applying this observation to the sets $\{ |x-x_i^\ve|; x\in B_j^\ve\}$, we find that there exists some $3\le C\le 3(N+1)(N+2)/2$ such that the discs $B_{C\ve}(x_i^\ve)$ cover $A_\ve$ and, in addition, $C_{C\ve}(x_i^\ve)$ does not intersect any of the $B_{\ve}(x_j^\ve)$. Therefore, we have $C_{C\ve}(x_i^\ve)\cap A_\ve=\emptyset$. We find that each $B_{C\ve}(x_i^\ve)$ is a bad disc. 

\smallskip
Finally, the existence of enlarged bad discs satisfying the additional properties \eqref{at211} or \eqref{at9b} is then obtained as in \cite{bbh}*{Chapter IV}. (This may require taking passing to a further subsequence.)

\smallskip
\noindent
{\it 2.} The proof is essentially the same, thanks to the upper bound \eqref{aq2}. The only difference arises from the argument leading to the existence of $N=N(\delta_1, \deg g)$, since we are not in position to rely on the assumption \eqref{as1}. Let $0<\beta<\alpha<1$ be fixed. Let $0<\lambda<1/2$ and let $\eta$ be as in Lemma \ref{ab8} (corresponding to this $\lambda$). If  $\ve$ is sufficiently small and $x\in A_\ve$, then $E_\ve(u_\ve, B_{\ve^\alpha/5}(x))\ge \eta |\ln\ve|$. Therefore, there exists some $N_1=N_1(\deg g, \delta_1)$ such that $A_\ve$ can be covered with at most $N_1$ admissible balls $B_{\ve^\alpha}(x_i^\ve)$. By a mean value argument, there exist a constant $C_0=C_0(\deg g, \alpha, \beta)$ and radii $r=r_i^\ve$ such that $\ve^\alpha\le r_i^\ve\le \ve^\beta$ and
\be
\l{at5}
r\int_{\p (B_r(x_i^\ve)\cap\O)} |u_\tau|^2+ \frac {r}{\ve^2}\int_{\p (B_r(x_i^\ve)\cap\O)} (1-|u|^2)^2\le C_0.
\ee

Combining \eqref{at5} with \eqref{as4}, we find that, for some constant $D=D(\deg g, \delta_1)$, we have, for small $\ve$, 
 \be
 \l{at6}
 \frac 1{\ve^2}\int_{B_{\ve^\alpha}(x_i^\ve)\cap\O}(1-|u|^2)^2\le D,\ \fo i.
 \ee
 
 Using \eqref{at6} and arguing as in the proof of  item {\it 1}, we find that $A_\ve\cap B_{\ve^\alpha}(x_i^\ve)$ can be covered with at most $N_2=N_2(\deg g, \delta_1)$ discs $B_{\ve}(x)$. This yields a covering of $A_\ve$ with at most $N_1N_2$ discs, each of radius $\ve$.
 
 \smallskip
 \noindent
 {\it 3}. By the argument used in the proof of item {\it 1}, we may actually choose some $C(\ve)\in [3, 3(N+1)(N+2)/2]$.
\end{proof}

\section{The energy is bounded on \enquote{intermediate} balls away from the bad discs}
\l{ba1}

In this section, we derive an easy consequence of the  results in Section \ref{ak2} combined with the bad discs structure. 
We consider the setting of Section  \ref{ab3}. 
{\it We assume \eqref{ak1}.}
Let
 $\O$ and the boundary datum $g\in C^\infty(\p\O ; \so)$ be fixed. {\it Let $u=u_\ve$ be a minimizer of $E_\ve$ in $H^1_g(\O ; \C)$.} By Lemma \ref{at1} item {\it 2}, we may find an integer $N$ and a finite constant $C$, depending only on $\deg g$, such that, for small $\ve$ (depending on $\O$ and $g$),  \eqref{at290} holds.  

 \smallskip
An inspection of the proofs in Section \ref{ab3} shows that the smallness of the constant $\eta$ plays a role mainly in the existence of a suitable radius $r$ such that, on $\p (B_r(x)\cap\O)$,  (i) $|u_\ve|$ is far away from $0$ and (ii) $u_\ve/|u_\ve|$ has degree $0$. This is especially useful in Lemmas \ref{ab6} and \ref{ac1}.  {\it If we know that the assumptions (i) and (ii) are valid for all \enquote{useful} radii $r$ (i.e., for the radii obtained {\it via} a mean value argument, as, e.g., in \eqref{ad3}), then  Lemmas \ref{ab6} and \ref{ac1} hold without the smallness assumption on $\eta$.} These considerations lead to the following result.

\bl
 \l{ba2}
 Let $0<\alpha<1<\beta<1/\alpha$ and $\ve<1/2$. Let $S_\ve$ denote the union of the  bad discs in \eqref{at290} and suppose that $B=B_R(x)$ is a ball such that
  \be
 \l{ba3}
\ve^{1/\beta}\le R\le \ve^\alpha\text{ and } \dist (x, S_\ve)\ge R.
 \ee
Then, for sufficiently small $\ve$ and  some finite constant $C_1=C_1(\deg g, \delta_1,\alpha, \beta)$, we have
 \be
 \l{ba4}
 E_\ve(u, B_{R^\beta}(x)\cap\O)\le C_1\text{ and }G_\ve(u, B_{R^\beta}(x)\cap\O)\le C_1.
 \ee 
 \el
 \bp[Sketch of proof]
 In what follows,  $C_j=C_j(\deg g, \delta_1, \alpha, \beta)$ is a  finite constant, and $\ve$ is sufficiently small. 
 By Lemma \ref{aq1}, the assumption $\ve^{1/\beta}\le R\le\ve^\alpha$,  and a mean value argument, there exists some  $R^\beta<r<R/2$ such that 
 \be
 \l{ba5}
 r\int_{\p B_r(x)\cap\O}|u_\tau|^2+\frac r{\ve^2}\int_{\p  B_r(x)\cap\O}(1-|u|^2)^2\le C_2.
 \ee
 
Let $C$ be the constant in \eqref{at290}.
Since $\dist (\overline B_r(x), S_\ve))\ge R/2\ge C\ve$, we find that $|u|\ge 1/2$ in $\overline B_r(x)\cap\overline\O$, and thus, in particular, $u/|u|$ has degree zero on $\p (B_r(x)\cap\O)$. By the above, we are now in position to repeat the proof of  \eqref{ag1} in Lemma \ref{ab6} and obtain the estimates
 \begin{gather}
 \l{ba6fg}
 \verti{\int_{C_r(x)\cap\O} u\wedge u_\tau}\le C_3\ \mbox{and}\ G_\ve(u, B_r(x)\cap\O)\le C_4.
 \end{gather}
 
 Combining \eqref{ba6fg} and   \eqref{af4}, we find that
 \be
 \l{ba7}
 E_\ve(u, B_r(x)\cap\O)\le C_5.
 \ee

 Finally, 
\eqref{ba4} follows from the second inequality in \eqref{ba6fg} and \eqref{ba7}.
 \ep

 In Section \ref{ca1}, we will encounter an avatar of the above considerations; see Lemma \ref{cd1} there.

\section{Zoom near the boundary}
\l{aw1}

In this section, we prove that the bad discs described in the previous section are far away from the boundary at the $\ve$ scale. This fact  is an obvious consequence of the following result.
\bl
\l{aw2}
Let $u=u_\ve$ be critical points of $E_\ve$ in $H^1_g(\O)$ satisfying the energy bound \eqref{aq4}. Let $C$ be a fixed constant. Consider, for each $\ve$, a point $y_\ve\in\O$ such that $\dist(y_\ve, \p\O)\le C\ve$ and let, for small $\ve$, $z_\ve$ be the nearest point projection of $y_\ve$ on $\p\O$. Then 
\be
\l{aw3}
u_\ve (y_\ve)-g(z_\ve)\to 0\ \text{as }\ve\to 0.
\ee
\el
\begin{coro}
\l{aw4}
Under the assumptions of Lemma \ref{at1}, the centers $x_\ve^j$ of the bad discs satisfy $\dist (x_\ve^j, \p\O)\gg\ve$ as $\ve\to 0$.
\end{coro}
\bp[Proof of Lemma \ref{aw2}] It suffices to obtain \eqref{aw3} up to a subsequence. This is obtained {\it via} a blow-up analysis. Consider the rescaled maps
\bes
v_\ve(z):=u_\ve(\ve z+z_\ve),\ \fo z\in U_\ve:=\ve^{-1}(-z_\ve+\O),
\ees
extended with the same formula to $\overline U_\ve$.

\smallskip
Note that $0\in \p U_\ve$. 
Up to a subsequence, we have $g(z_\ve)\to \tilde{g}$ for some constant $\tilde{g}\in\so$, and the  unit outer normal to $U_\ve$ at the origin converges to some $\xi\in\so$. We work with such a subsequence. Consider the half-plane
$H:=\{ X\in\R^2;\, X\cdot\xi<0\}$. 

\smallskip
We next note that, by Corollary \ref{aw9}, $v_\ve$ has bounded derivatives, at any order. Moreover, the tangential derivative of $v_\ve$ on $\p U_\ve$ is (uniformly) of the order of $\ve$. By the above, possibly up to a further subsequence, we have $v_\ve\to v$ pointwise in $H$ and uniformly on bounded sets, where the map $v$ is smooth in $\overline H$, satisfies $v=\tilde{g}$ on $\p H$, and is a solution of 
\be
\l{aw10}
\L v=v(1-|v|^2) \ \text{in }H.
\ee
(To be specific, uniform convergence on bounded sets means that
\be
\l{aw11}
\max\{ |v_\ve(x)-v(x)|; x\in K\cap \overline U_\ve\}\to 0 \text{ as }\ve\to 0,\ \text{$\fo K\subset \overline H$ compact}.)
\ee

We note that the conclusion of the lemma amounts to $v=\tilde{g}$. In order to obtain this conclusion, we first 
 establish an additional property of $v$. By Lemma \ref{ap1}, the energy bound \eqref{aq4} and a mean value argument, there exist a finite constant $C$ (possibly depending on $\delta_1$, $K$, $\O$, and $g$) and  some $r=r_\ve \in (\ve^{1/2}, \ve^{1/3})$ such that
\be
\l{aw13}
r\int_{B_r(z_\ve)\cap\p\O}|u_{\ve,\nu}|^2\le C.
\ee

Estimate \eqref{aw13} is equivalent to
\be
\l{aw14}
\int_{B_{r/\ve}(0)\cap \p U_\ve}|v_{\ve,\nu}|^2\le C\frac\ve r.
\ee

Using \eqref{aw14}, Corollary \ref{aw9}, and the fact that $\ve/r\to 0$ as $\ve\to 0$, we find that $v$ satisfies
\be
\l{aw16}
\begin{cases}
v\in C^\infty(\overline H)
\\
\L v=v(1-|v|^2)&\text{in }H
\\
v=\tilde{g}\in\so&\text{on }\p H
\\
v_\nu=0&\text{on }\p H
\end{cases}.
\ee

We complete the proof of lemma {\it via} the following 

\smallskip
\noindent
{\it Claim. The only solution of \eqref{aw16} is $v\equiv\tilde{g}$.} In order to prove the claim, we extend $v$ to $\R^2$ with the value $\tilde{g}$ on $H_-:=\R^2\setminus\overline H$, and still denote the extension by $v$. We note that $v\in H^1_{loc}(\R^2)$ and that $\L v=v(1-|v|^2)$ in $H$ and also in $H_-$. The key observation is that we actually have 
\be
\l{aw18}
\L v=v(1-|v|^2)\ \text{in the weak sense in }\R^2.
\ee

Indeed, since  $v\equiv\tilde{g}=$constant on $\p H$ and $v_\nu=0$ on $\p H$, we find that  $\na v=0$ on $\p H$. Therefore,  if $\va\in C^\infty_c(\R^2 ; \R^2)$ and we write $\tilde{g}=(\tilde{g}^1, \tilde{g}^2)$, $\va=(\va^1, \va^2)$, then (using the fact that $\L$ is formally self-adjoint)
\begin{gather*}
\int_H v\cdot {}^t\L\va=\int_H v\cdot \L\va=T+\int_H \L v\cdot \va=T+\int_H \va\cdot v(1-|v|^2),
\\
\int_{H_-}v \cdot {}^t \L\va=\int_{H_-}v\cdot \L\va=-T+\int_{H_-} \L v\cdot \va=-T=-T+\int_{H_-} \va\cdot v(1-|v|^2),
\\
\int_{\R^2} v\cdot {}^t \L\va=\int_{\R^2} v\cdot \L\va=\int_{\R^2}\va\cdot v(1-|v|^2),
\end{gather*}
where $T$ is the boundary term
\bes
\ba
T=\int_{\p H}[&-K_1\tilde{g}^1\nu_x (\va^1_x+\va^2_y)-K_3\tilde{g}^1\nu_y(\va^1_y-\va^2_x) 
\\
&-K_1\tilde{g}^2 \nu_y (\va^1_x+\va^2_y)+K_3\tilde{g}^2\nu_x (\va^1_y-\va^2_x)],
\ea
\ees
whence \eqref{aw18}.

\smallskip
We complete the proof of the claim, and thus of the lemma, by noting that the definition of $v$ on $H_-$ combined with 
Lemma \ref{av2} implies that $v\equiv\tilde{g}$. 
 \ep

\section{Zoom of enlarged bad discs}
\l{bb1}

{\it The results in this section are valid under the assumption \eqref{ak1}. 
We consider maps  $u=u_\ve\in H^1_g(\O)$  satisfying the assumptions  of Lemma \ref{at1}, i.e.: (i) either critical points of $E_\ve$ in $H^1_g(\O)$ in a strictly star-shaped domain $\O$, satisfying the  {\it a priori} bound \eqref{aq4}; (ii) or minimizers of $E_\ve$ in $H^1_g(\O)$. } In particular, the conclusions of Lemma \ref{at1} hold.

\smallskip
Consider an enlarged bad disc as in Lemma \ref{at1}. Consider the rescaled maps $v_\ve$ as in the proof of Lemma \ref{aw2}, with $z_\ve$ replaced with the center of the bad disc. By Lemma \ref{aw6}, Corollary \ref{aw4} and the {\it a priori}  estimates \eqref{aqa}, \eqref{aqb}, and \eqref{aw8a} (all valid, for small $\ve$, as a consequence of the assumptions considered above), the following hold, with constants independent of small $\ve$, $\delta$ satisfying \eqref{ak1}, the boundary datum $g$, and with $R_\ve\to\infty$ as $\ve\to 0$:
\begin{gather}
\l{bc1}
v_\ve\ \text{is defined in }B_{R_\ve}(0),
\\
\l{bc2}
|v_\ve(0)|\le 1/2,
\\
\l{bc3}
|D^k v_\ve(x)|\le \widetilde C_k\ \text{in }B_{R_\ve}(0),\ \fo k,
\\
\l{bc4}
\L v_\ve =v_\ve(1-|v_\ve|^2)\ \text{in }B_{R_\ve}(0).
\end{gather}

Moreover, there exists  finite constants $D_1$, $D_2$ (possibly depending on $g$ if $u_\ve$ is merely a critical point of $E_\ve$) such that
\begin{gather}
\l{bc5}
\int_{B_{R_\ve}(0)}(1-|v_\ve|^2)^2\le D_1,
\\
\l{bc5a}
|v_\ve(x)|\ge 1/2\ \text{if }D_2\le |x|<R_\ve.
\end{gather}

In addition, if $u_\ve$ is a minimizer of $E_\ve$, 
\be
\l{bc6}
v_\ve\ \text{is a minimizer of $E_1$ in $B_{R_\ve}(0)$ with respect to its own boundary condition}.
\ee

It follows that, possibly up to a subsequence, $(v_\ve)$ converges in $C^\infty_{loc}(\R^2)$ to a smooth  map $v:\R^2\to\C$ such that
\begin{gather}
\l{bd1}
|v(0)|\le 1/2,
\\
\l{bd2}
|D^k v|\le \widetilde C_k,\ \fo k,
\\
\l{bd3}
\L v=v(1-|v|^2),
\\
\l{bd4}
\int_{\R^2}(1-|v|^2)^2<\infty,
\end{gather}
and, if $u_\ve$ is a minimizer of $E_\ve$,
\be
\l{bd5}
v\ \text{is an entire local minimizer of $E_1$ (in the sense of De Giorgi)}.
\ee

For further use, let us note that any map $v$ satisfying \eqref{bd2} and \eqref{bd4} satisfies $\lim_{|x|\to\infty}|v(x)|=1$, and thus $v$ has a \enquote{degree at $\infty$}, in the sense that, for large $R$ (depending on $v$), the integer $\deg  v:=\deg (v/|v|, C_R(0))$ is well-defined and independent of $R$.

\smallskip
In what follows, we derive some easy consequences of the analysis developed up to now.

\begin{coro}
\l{be1}
For every $\delta$, there exists a bounded  entire local minimizer of $E_1$ satisfying \eqref{bd4} and of negative degree. 
\end{coro}
\bp
Consider any domain $\O$ and any boundary datum of negative degree. The enlarged bad discs in Lemma \ref{at1} satisfy 
\bes
\sum_j \deg (u_\ve/|u_\ve|, C_{C'\ve}(x_j^\ve))=\deg (g, \p\O)
\ees
(since $u_\ve$ does not vanish in $\overline\O\setminus\cup_j B_{C'\ve}(x^j_\ve)$). 
Therefore, at least one of them has negative degree on $C_{C'\ve}(x_j^\ve)$. Blowing-up this bad disc and possibly after passing to a subsequence, we obtain a $v$ as in the above statement.
\ep

\br
\l{be2}
With more work, it is possible to remove the assumption \eqref{bd4} and establish the following analogue of the main result in Sandier \cite{sandier}. Let $v$ be a bounded  entire local minimizer of $E_1$. Then $\d\int_{\R^2}(1-|v|^2)^2<\infty$.
However, it is unclear how to remove the boundedness assumption on $v$.
\er

We next note a first \enquote{small $\delta$} result.
\bl
\l{be3} We fix a boundary datum $g\in C^\infty(\p\O ; \so)$. 
Let $u_\ve$ be a minimizer of $E_{\ve}$ in $H^1_g(\O ; \C)$. Let $0<\lambda<2\pi$. There exist finite constants $\delta_0$, $C_1$, $C_2$ depending only on $\lambda$, such that, if $\delta<\delta_0$ and $\ve<\ve_0(\delta, \lambda)$, and $v_\ve$ is as above,  then:
\begin{gather}
\l{be4}
\int_{B_{C_1}(0)}(1-|v_\ve|^2)^2\ge \lambda,
\\
\l{be5}
|v_\ve(x)|\ge 1/2\text{ if } C_2\le |x|\le R_\ve,
\\
\l{be6}
\deg (v_\ve/|v_\ve|, C_{C_2}(0))=\pm 1.
\end{gather}
\el

\bc
\l{be8}
If $\delta<\delta_0$ and $u_\ve$ minimizes $E_\ve$ in $H^1_g(\O ; \C)$, then we may replace, in the definition of the enlarged bad discs, the condition $|u_\ve(x^j_\ve)|\le 1/2$ with $u_\ve(x^j_\ve)=0$.
\ec
\bp
This follows by noting that, by \eqref{be5} and \eqref{be6}, $v$ has to vanish in $B_{C_2}(0)$.
\ep

Combining Lemma \ref{be3} with the proof of Corollary \ref{be1}, we obtain the following
\bc
\l{be9}
If $\delta<\delta_0$, then there exists an entire local minimizer satisfying \eqref{bd4} and of degree $-1$.
\ec

\bp[Proof of Lemma \ref{be3}]
For a fixed $\delta$, consider any $v^\delta$ arising as a $C^\infty_{loc}(\R^2)$ limit of $v_\ve$ (possibly up to a subsequence $\ve_k\to 0$). The conclusion of the lemma follows provided any such $v^\delta$ has, for small $\delta$ and with respect to appropriate constants $C_1$ and $C_2$, the following properties: 
\begin{gather}
\l{bf1}
\int_{B_{C_1}(0)}(1-|v^\delta|^2)^2> \lambda,
\\
\l{bf2}
|v^\delta(x)|\ge 2/3\text{ if }  |x|\ge C_2,
\\
\l{bf3}
\deg (v^\delta/|v^\delta|, C_{C_2}(0))=\pm 1.
\end{gather}

In turn, \eqref{bf1}--\eqref{bf3} hold if any $v$ arising as a $C^\infty_{loc}(\R^2)$ limit of $v^\delta$ (possibly up to a subsequence) satisfies 
\begin{gather}
\l{bf4}
\int_{B_{C_1}(0)}(1-|v|^2)^2> \lambda,
\\
\l{bf5}
|v(x)|\ge 3/4\text{ if }  |x|\ge C_2,
\\
\l{bf6}
\deg (v/|v|, C_{C_2}(0))=\pm 1.
\end{gather}

In order to prove \eqref{bf4}--\eqref{bf6}, we note that \eqref{bd1} and \eqref{bd5} applied to $v^\delta$ yield, by letting  $\delta\to 0$, that  
\be
\l{bf8}
|v(0)|\le 1/2
\ee
and $v$ is an entire  local minimizer of $G_1$. (Here, we use the fact that, when $\delta=0$, the minimization of $E_1$ is equivalent to the minimization of $G_1$; see the proof of Lemma \ref{af1}.) Such minimizers are either constants of modulus $1$ (which cannot happen in our case, by \eqref{bf8}) or, up to a rotation and translation,  of the form $v(r e^{\im\theta})=f(r) e^{\pm \im\theta}$ \cite{radiale}. Here, $f\ge 0$ is (strictly) increasing and uniquely determined by the equation $-\Delta v=v(1-|v|^2)$ and the condition $f(r)\to 1$ as $r\to\infty$. Moreover, for such $v$ we have, by a straightforward application of Pohozaev's identity,
\be
\l{bf7}
\int_{\R^2}(1-|v|^2)^2=2\pi
\ee
(see also Brezis, Merle, and Rivi\`ere \cite{bmr} for a more general result). If, for $0<t<1$, we let $r_t$ be the unique solution of $f(r_t)=t$, then, by \eqref{bf7} and the monotonicity of $f$, \eqref{bf4} holds for large $C_1$, while (using \eqref{bf8}) \eqref{bf5} and \eqref{bf6} hold with $C_2:=r_{3/4}+r_{1/2}$.
\ep

\section{Small \texorpdfstring{$\delta$}{delta} analysis. Bad discs structure for minimizers}
\l{bg0}

Throughout this section, we consider minimizers $u=u_\ve$ of $E_\ve$ in $H^1_g(\O; \C)$, with boundary datum of degree $-D<0$. 
The main result of this section is the following.
\bt
\l{bg1}
There exists some $0<\delta_2<1$, possibly depending on $D$, but not on $\O$ or $g$,  such that, if $0\le\delta\le\delta_2$ and $\ve$ is small, then $u$ has exactly $D$ enlarged bad discs, all of degree $-1$.
\et

The proof of Theorem \ref{bg1}, which is somewhat similar to \cite{sandier_serfaty}*{Proof of Theorem 5.4},  is slightly easier in the case where $\O$ is strictly star-shaped. We start with this case, and later we present the minor modifications to be made in order to treat the general case. A first key ingredient is the following straightforward variant of Lemma \ref{an6}  combined with Lemma \ref{as5}.
\bl
\l{bg2}
Let $0\le\delta\le\delta_2<1$. 
If $u$ is a critical point of $E_\ve$ in $H^1_g(\O)$, then there exists some finite constant $C$ depending only on $\delta_2$ such that, for every $x\in\O$ and $r>0$, 
\be
\l{bg3}
\ba
\frac 1{\ve^2}\int_{B_r(x)\cap \O}(1-|u|^2)^2+(1-\delta_2)r\int_{C_r(x)\cap\O}|u_\nu|^2\le &\frac r{2\ve^2}\int_{C_r(x)\cap\O}(1-|u|^2)^2
\\
&+(1+\delta_2)r\int_{C_r(x)\cap\O}|u_\tau|^2
\\
&+C r\int_{B_r(x)\cap\p\O}|\na u|^2.
\ea
\ee

If, moreover, $\O$ is strictly star-shaped, then there exists some finite constant $\widetilde C$, depending on $\delta_2$, $\O$ and $g$, such that, for every $x\in\O$ and $r>0$, 
\be
\l{bg4}
\ba
\frac 1{\ve^2}\int_{B_r(x)\cap \O}(1-|u|^2)^2+(1-\delta_2)r\int_{C_r(x)\cap\O}|u_\nu|^2\le &\frac r{2\ve^2}\int_{C_r(x)\cap\O}(1-|u|^2)^2
\\
&+(1+\delta_2)r\int_{C_r(x)\cap\O}|u_\tau|^2
\\
&+\widetilde C r,
\ea
\ee
and therefore
\be
\l{bg4bis}
\ba
\int_{C_r(x)\cap\O}&\left[\frac12|\na u|^2+\frac{1}{4\ve^2}(1-|u|^2)^2\right]
\\
\ge &\frac 1r\bigg[\frac 1{2(1+\delta_2)\ve^2}\int_{B_r(x)\cap\O} (1-|u|^2)^2-\frac{\widetilde C r}{2(1+\delta_2)}\bigg].
\ea
\ee
\el 

A second ingredient is reminiscent of the expanding balls technique in \cite{sandier_serfaty}*{Chapter 4}. Although we could adapt to our context the more general arguments in \cite{sandier_serfaty}, we rely on a very simple result, sufficient for our purposes. Since the proof does not \enquote{see} the space dimension, we state the result in $\R^n$ (and use it in $\R^2$).

\bl
\l{bh1} 
Let $n\ge 2$, $R_0>0$, and $X\subset\R^n$. Set $U:=\{ x\in\R^n;\, \dist (x, X)\le R_0\}$. Consider an integer $N$, a radius $0<R\le 3^{1-N}R_0$, and $N$ (not necessarily disjoint) closed balls  of radius $R$, $B_j=\overline B_R(x_j)$, $1\le j\le N$, such that $x_j\in X$, $\fo j$. For each $x\in X$ and $0<r\le R_0$, set
\bes
J(x,r):=\{ j;\, B_j\subset B_r(x)\}.
\ees

Let $\lambda_1,\ldots,\lambda_N\ge 0$. 
Suppose that a  non-negative Borel  function $h$ defined on $U\setminus\cup_j B_j$ has  the following property:
\be
\l{bh2}
[S_r(x)\cap B_j=\emptyset, \, \fo j]\implies \int_{S_r(x)}h\ge \frac 1{r}\sum_{j\in J(x,r)}\lambda_j ,\ \fo x\in X,\, \fo 0<r\le R_0.
\ee

Then
\be
\l{bh3}
\int_{U\setminus\cup_j B_j}h\ge \left(\ln \frac {R_0}{3^{N-1}R}\right)\sum_j \lambda_j .
\ee
\el
\bp[Proof of Lemma \ref{bh1}] With no loss of generality, we may assume that $R_0=1$. The proof is by induction on $N$. When $N=1$, we note that
\bes
\int_{U\setminus B_R(x_1)}h\ge \int_{B_1(x_1)\setminus B_R(x_1)}h=\int_R^1 \int_{S_r(x_1)}h(y)\, d\sigma(y)dr\ge \left(\ln\frac 1R\right) \lambda_1,
\ees
whence the conclusion.

\smallskip
Assume next that \eqref{bh3} holds for $(N-1)$ balls and consider $N$ balls as in the statement of the lemma. Set $\d {\bf m}:=\frac 12\min_{i\neq j} |x_i-x_j|$.

\smallskip
\noindent
{\it Case 1. We have ${\bf m}\le R$.} Equivalently, we have $B_i\cap B_j\neq\emptyset$ for some $i\neq j$. With no loss of generality, we may assume that $B_{N-1}\cap B_N\neq\emptyset$. Consider the balls $\widetilde B_j:=\overline B_{3R}(x_j)$, $1\le j\le N-1$. Then $\cup_{1\le j\le N}B_j\subset \cup_{1\le j\le N-1}\widetilde B_j$ and $3R\le 3^{1-(N-1)}$. Associate with these balls the numbers $\widetilde\lambda_1:=\lambda_1,\ldots, \widetilde\lambda_{N-2}:=\lambda_{N-2},  \widetilde\lambda_{N-1}:=\lambda_{N-1}+\lambda_N$. Then clearly the radius $3R$,  the $(N-1)$ balls $\widetilde B_j$,   and the $(N-1)$ numbers $\widetilde\lambda_j$ satisfy the adapted version of \eqref{bh2}. By the induction assumption, we find that 
\bes
\int_{U\setminus \cup_j B_j}h\ge \int_{U\setminus \cup_j \widetilde B_j}h\ge \left(\ln \frac 1{3^{N-2}(3 R)}\right)\sum_j  \widetilde\lambda_j=\left(\ln \frac 1{3^{N-1} R}\right)\sum_j  \lambda_j,
\ees 
whence the desired conclusion in Case 1.

\smallskip
\noindent
{\it Case 2. We have $R<{\bf m}\le 3^{1-N}$.} Consider the balls $\overline B_{\bf m}(x_j)$, $1\le j\le N$. Then (by definition of ${\bf m}$) two of these balls have a common point. By the conclusion of Case 1, we have
 \be
 \l{bh4}
\int_{U\setminus \cup_j B_{\bf m}(x_j)}h \ge \left(\ln \frac 1{3^{N-1}{\bf m}}\right)\sum_j \lambda_j.
\ee

On the other hand, we have (using \eqref{bh2})
\be
\l{bh5}
\int_{\cup_j (B_{\bf m}(x_j)\setminus B_R(x_j))} 
= \sum_j \int_R^{\bf m} \int_{S_r(x_j)}h(y)\, d\sigma(y)dr\ge \left(\ln \frac {\bf m}R\right)\sum_j\lambda_j.
\ee

We complete Case 2 by combining \eqref{bh4} and \eqref{bh5}.

\smallskip
\noindent
{\it Case 3. We have ${\bf m}> 3^{1-N}$.} Then the balls $\overline B_{3^{1-N}}(x_j)$ are mutually disjoint and therefore (using \eqref{bh2})
\bes
\ba
\int_{U\setminus\cup_j B_j}h\ge &\sum_j\int_{B_{3^{1-N}}(x_j)\setminus B_R(x_j)}h=\sum_j \int_R^{3^{1-N}}\int_{S_r(x_j)}h(y)\, d\sigma(y)dr
\\
\ge &\left(\ln \frac 1{3^{N-1}R}\right)\sum_j \lambda_j.
\ea
\qedhere
\ees
\ep

\bp[Proof of Theorem \ref{bg1} when $\O$ is strictly star-shaped] Let $\lambda<2\pi$ be a number to be fixed later (sufficiently close to $2\pi$). Let $\delta<\delta_0=\delta_0(\lambda)$, with $\delta_0$ as at the end of Section \ref{bb1}. By Lemma \ref{be3}, if we prove that there are {\it at most} $D$ enlarged bad discs, then there are {\it exactly} $D$ enlarged bad discs, and their respective degrees are all $-1$.

\smallskip
Let $C_1$ be as in Lemma \ref{be3}. (Note that $C_1$ depends only on $\lambda$.) Let $x_\ve^j$, $1\le j\le N(\ve)\le N$,  be the centers of the enlarged bad discs (as in Lemma \ref{at1} item {\it 2}). For a sufficiently small $\ve$, the enlarged bad discs are contained in $\O$ (Corollary \ref{aw4}) and, by Lemma \ref{be3} and the convergence results derived at the beginning of Section \ref{bb1},  we have (after rescaling back the functions $v_\ve$) 
\be
\label{bk1}
\liminf_{\ve \to 0}\inf_j\frac 1{\ve^2}\int_{B_{C_1\ve}(x^j_\ve)}(1-|u_\ve|^2)^2\ge \lambda.
\ee

\smallskip
Consider some smooth $\so$-valued extension $G$ of $g$ to $\R^2\setminus\overline\O$. (Recall that, for simplicity, we have assumed $\O$ simply connected, and therefore such an extension does exist.) We still denote $u=u_\ve$ the extension of $u$ with the value $G$ outside $\overline\O$.

\smallskip
Consider a small number $a>0$, to be fixed later. Set 
\begin{gather}
\l{bj1}
X:=\overline\O,\ U:=\{ x\in\R^2;\, \dist (x, \O)\le 1\},
\\
\l{bj2}
R:=C_1\ve,\ B_j:=\overline B_R(x^j_\ve),\, 1\le j\le N(\ve),
\\
\l{bj3}
h:=\frac 12|\na u|^2+\frac 1{4\ve^2}(1-|u|^2)^2,
\\
\l{bj4}
\lambda_1=\cdots=\lambda_{N(\ve)}:=\frac \lambda{2(1+\delta_2)}-a.
\end{gather}

Let $a$ be sufficiently small such that $\lambda_1>0$ and
\be
\label{bj4a}
R_0:=\frac{2(1+\delta_2) a}{\widetilde C}\le 1.
\ee

By  \eqref{bk1} and \eqref{bg4bis}, when $x\in\O$, $0<r\le R_0$, and $B_r(x)$ contains at least one enlarged bad disc $B_{j_0}$, we have
\be
\label{bj4b}
\int_{C_r(x)}h\ge \frac 1r\left[\sum_{j\in J(x,r)}\frac \lambda{2(1+\delta_2)}-\frac {\widetilde C R_0}{2(1+\delta_2)}\right] \ge \frac 1r\sum_{j\in J(x,r)}\lambda_j.
\ee

Using \eqref{bj4b} and 
 Lemma \ref{bh1} (with $R_0$ given by \eqref{bj4a}), we find that, for sufficiently small $\ve$, we have
\be
\l{bj4c}
\int_U h\ge \left[\frac{\lambda}{2(1+\delta_2)}-a\right]N(\ve)\ln \frac {R_0}{C_1\ve},
\ee
and therefore (using the definition of $h$ and the fact that $G$ is smooth and $\so$-valued), the Gin\-zburg-Landau energy of $u$ satisfies
\be
\l{bi1}
G_\ve(u,\O)\ge \left[\frac{\lambda}{2(1+\delta_2)}-a\right]N(\ve)\ln \frac 1\ve- C(a,\ve).
\ee
Here, $C(a,\ve)$ is bounded as $\ve\to 0$.

\smallskip
On the other hand, Lemma \ref{aq1} combined with Lemma \ref{af1} yields, with a constant $C_3$ depending on $\delta_2$, $\O$, and $g$, the following bound for the standard Ginzburg-Landau energy:
\be
\l{bg7}
G_\ve (u, \O)\le \frac \pi{1-\delta_2} D \ln\frac 1\ve+C_3,\ \fo 0<\ve\le 1.
\ee

We finally choose $\lambda$, $a$,  and $\delta_2$ such that
\be
\l{bi2}
(D+1)\left[\frac{\lambda}{2(1+\delta_2)}-a\right]>D\frac\pi{1-\delta_2}.
\ee
(This is possible, provided $\lambda$ is sufficiently close to $2\pi$ and $\delta_2$ and $a$ are sufficiently small.) 

\smallskip
By \eqref{bi1}, \eqref{bg7}, and \eqref{bi2}, for small $\ve$ we have $N(\ve)<D+1$, i.e., $N(\ve)\le D$.
\ep

\bp[Proof of Theorem \ref{bg1} in the general case] Let $r_0$ be as in Lemma \ref{ap1}. Let $x_0\in\p\O$. Using the upper bound \eqref{bg7}, \eqref{as4}, and a mean value argument, we find that (with  $C_1$ as in \eqref{as4}, $C_4=C_4(D, \delta_2)$, and $\ve\le\ve_0=\ve_0(D)$) there exists some $r_0/2<r<r_0$ such that
 \bes
 C_1 r\int_{B_r(x_0)\cap\p\O}|u_\nu|^2\le C_4 r \ln \frac  1\ve.
 \ees

Covering $\p\O$ with a finite number (independent of $\ve\le \ve_0$) of balls $B_r(x_0)$ as above, we find that
\be
\l{bi3}
\int_{\p\O}|u_\nu|^2\le C_5  \ln\frac 1\ve.
 \ee
 
 Combining \eqref{bi3} with \eqref{bg3}, we obtain, for small $\ve$,  the following versions of \eqref{bg4} and \eqref{bg4bis}:
 \be
\l{bi4}
\ba
\frac 1{\ve^2}\int_{B_r(x)\cap \O}(1-|u|^2)^2+(1-\delta_2)\int_{C_r(x)\cap\O}|u_\nu|^2\le &\frac r{2\ve^2}\int_{C_r(x)\cap\O}(1-|u|^2)^2
\\
&+(1+\delta_2)r\int_{C_r(x)\cap\O}|u_\tau|^2
\\
&+\widetilde C r \ln\frac 1\ve
\ea
\ee
and
\be
\l{bi4bis}
\ba
\int_{C_r(x)\cap\O}&\left[\frac 12 |\na u|^2+\frac 1{4\ve^2}(1-|u|^2)^2\right]
\\
\ge &\frac 1r\frac 1{2(1+\delta_2)\ve^2}\int_{B_r(x)\cap\O} (1-|u|^2)^2-\frac{\widetilde C r \d \ln (1/\ve)}{(1+\delta_2)r}.
\ea
\ee

Set 
\be
\l{bk2}
R_0:=\frac {(1+\delta_2)a}{\widetilde C \ln (1/\ve)}.
\ee

For small $\ve$, we have $R_0\le 1$. Repeating the argument in the star-shaped case and using \eqref{bi4bis}, we see that, with $R_0$ as in \eqref{bk2}, \eqref{bj4b} still holds. Therefore, we are in position to derive the analogue of \eqref{bi1}, which in our case (using \eqref{bk2}) is
\be
\l{bi1bis}
G_\ve(u,\O)\ge \left[\frac{\lambda}{2(1+\delta_2)}-a\right]N(\ve)\ln \frac 1\ve- C(a,\ve)-\widetilde C(a,\ve)\ln \ln \frac 1\ve,
\ee
with $C(a,\ve)$ and $\widetilde C(a,\ve)$ bounded as $\ve\to 0$. 
Finally, we choose $\lambda$, $a$, and $\delta_2$ as in \eqref{bi2} and complete the proof {\it via} \eqref{bg7}, \eqref{bi2}, and \eqref{bi1bis}.
\ep

\section{Small \texorpdfstring{$\delta$}{delta} analysis. Insight on the locations of the bad discs}
\l{bm1}

Throughout this section, we consider minimizers $u=u_\ve$ of $E_\ve$ in $H^1_g(\O; \C)$, with boundary datum of degree $-D<0$.  {\it We let $\delta_2$ be as in Theorem \ref{bg1}.} 
The main result of this section is the following theorem that will subsequently be sharpened in several directions in Sections \ref{bu1} and \ref{ca1}. 
\bt
\l{bm2}
Let $0<\alpha<1$. 
There exists some $0<\delta_3\le\delta_2$, possibly depending on $D$ and $\alpha$, but not on $\O$ or $g$,  such that, if $0\le\delta\le\delta_3$ and $\ve$ is small, then the centers $x^j_\ve$, $j=1,\ldots, D$, of the enlarged bad discs satisfy
\begin{gather}
\l{bm3}
 {\bf m}:=\frac 12\min_{j\neq k}|x^j_\ve-x^k_\ve|\ge \ve^\alpha,
 \\
 \l{bm4}
 \dist(x^j_\ve, \p\O)\ge \ve^\alpha,\ \fo j.
\end{gather}
\et

\smallskip
The proof of \eqref{bm3} relies on the following generalization of Lemma \ref{bh1}.
\bl
\l{bm5}
We use the same notation as in Lemma \ref{bh1}. 
Let $\Phi:\R_+\to\R_+$ be a superadditive function. Let $N\ge 2$ and  $\lambda_1,\ldots,\lambda_N\ge 0$.  Consider the numbers
\begin{gather*}
b=b(\lambda_1,\ldots,\lambda_N):=\min_{i\neq j}[\Phi(\lambda_i+\lambda_j)-\Phi(\lambda_i)-\Phi(\lambda_j)]\ge 0, 
\\ 
{\bf m}:=\frac 12\min_{i\neq j}|x_i-x_j|,\  \rho:=\max (R, {\bf m}).
\end{gather*}

Suppose that a  non-negative Borel  function $h$ on $U\setminus\cup_j B_j$ has  the following property:
\be
\l{bm6}
\begin{split}
[S_r(x)\cap B_j=\emptyset, \, \fo j]\implies \int_{S_r(x)}h\ge \frac 1{r}\Phi\left(\sum_{j\in J(x,r)}\lambda_j \right),
\\
\ \fo x\in X,\,  \fo 0<r\le R_0.
\end{split}
\ee

Then
\be
\l{bm7}
\int_{U\setminus\cup_j B_j}h\ge \left(\ln \frac {R_0}{3^{N-1}R}\right)\sum_j \Phi(\lambda_j)+b \ln \frac{R_0}{3^{N-1}\rho}.
\ee
\el

Note that one recovers Lemma \ref{bh1} by taking $\Phi=\text{Id}$.
\bp[Proof of Lemma \ref{bm5}] We may assume that $R_0=1$. The proof is by induction on $N$. We mainly rely on Lemma \ref{bh1}, using the fact that, thanks to the superadditivity of $\Phi$, the assumption \eqref{bh2} is satisfied with $\lambda_j$ replaced with $\Phi(\lambda_j)$. The case $N=1$ is a special case of Lemma \ref{bh1} if we take by convention $b=0$ when $N=1$. Assuming that the result holds for $(N-1)$ balls, we argue as in the proof of Lemma \ref{bh1}. 

\smallskip
\noindent
{\it Case 1. We have ${\bf m}\le R$ (and thus $\rho=R$).} Consider the enlarged balls $\widetilde B_j$ as in Case 1 in the proof of Lemma \ref{bh1}. Using the conclusion of Lemma \ref{bh1} (with $\lambda_j$ replaced with $\Phi(\lambda_j)$), we find that
\begin{align*}
\int_{U\setminus \cup_j B_j}h\ge &\left(\ln \frac 1{3^{N-1} R}\right)\left(\sum_{j\le N-2}\Phi (\lambda_j)+\Phi (\lambda_{N-1}+\lambda_N)\right)
\\
=& \left(\ln \frac 1{3^{N-1} R}\right)\left(\sum_{j}\Phi (\lambda_j)+[\Phi(\lambda_{N-1}+\lambda_N)-\Phi(\lambda_{N-1})-\Phi(\lambda_N]\right)
\\
\ge  & \left(\ln \frac 1{3^{N-1} R}\right)\sum_{j}\Phi (\lambda_j)+b \ln \frac 1{3^{N-1} R}
\\
=  & \left(\ln \frac 1{3^{N-1} R}\right)\sum_{j}\Phi (\lambda_j)+b \ln \frac 1{3^{N-1} \rho}.
\end{align*}
{\it Case 2. We have $R<{\bf m}\le 3^{1-N}$.} Arguing as in Case 2 in the proof of Lemma \ref{bh1} and using the conclusion of Case 1 above, we find that
\bes
\ba
\int_{U\setminus \cup_j B_j}h\ge &\left(\ln \frac 1{3^{N-1}{\bf m}}\right)\sum_j \Phi(\lambda_j)+b\ln \frac 1{3^{N-1}{\bf m}}+\left(\ln \frac {\bf m}R\right)\sum_j \Phi(\lambda_j)
\\
=&\left(\ln \frac 1{3^{N-1}R}\right)\sum_j \Phi(\lambda_j)+b\ln \frac 1{3^{N-1}\rho}.
\ea
\ees
{\it Case 3. We have ${\bf m}>3^{1-N}$.} The conclusion follows from Lemma \ref{bh1} noting that in this case we have $\d b\ln\frac 1{3^{N-1}\rho}\le 0$.
\ep

We may now proceed with the proof of \eqref{bm3}. As for Theorem \ref{bg1}, the argument is slightly easier when $\O$ is strictly star-shaped, and we start with this case.
 
\bp[Proof of \eqref{bm3} when $\O$ is strictly star-shaped] Thanks to the assumption $\delta\le\delta_2$,   for small $\ve$ the function $u$ has exactly $D$ enlarged bad discs, each of degree $-1$ (Theorem \ref{bg1}). We extend $u$ to $\R^2$ as in the proof of Theorem \ref{bg1}. Let $X$, $U$, $R$, and $B_j$ be as in 
\eqref{bj1}--\eqref{bj2} (with $N(\ve)=D$). Let 
\begin{gather}
\l{bm8}
w:=\frac u{|u|},\ h:=\frac 12 |\na w|^2\ \text{in }U\setminus\cup_j B_j,
\\
\l{bm9}
\lambda_1=\ldots=\lambda_D:=1,
\\
\l{bm10}
\Phi(t):=\pi t^2, \ \fo t\in\R.
\end{gather}

By Theorem \ref{bg1}, we have
\be
\label{bn1}
\deg (w, C_r(x))=-\# J(x,r)\ \text{if }x\in X,\, 0<r\le 1, \text{ and }C_r(x)\cap B_j=\emptyset,\ \fo j.
\ee

For $x$ and $r$ as above, we have, by  the Cauchy-Schwarz inequality and \eqref{bn1}, 
\be
\l{bn2}
2\pi r \int_{C_r(x)}|w_\tau|^2\ge \left(\int_{C_r(x)}|w_\tau|\right)^2\ge (2\pi \# J(x,r))^2,
\ee
and thus \eqref{bm6} holds with  $\lambda_j$ and $\Phi$ as above.

\smallskip
We use the notation in Lemma \ref{bm5}. By Lemma \ref{bm5},  using the fact that $w$ is smooth and fixed outside $\O$ and that, by the construction of the enlarged bad discs, we have $\rho={\bf m}$ for small $\ve$, we find that 
\be
\l{bn3}
\frac 12\int_{\O\setminus \cup_j B_j} |\na w|^2\ge \pi D \ln \frac 1\ve+2\pi\ln \frac 1{\bf m}-C,
\ee
where $C$ is a finite constant independent of small $\ve$.

\smallskip
On the other hand, one easily checks the following:
\be
\l{bn4}
|u|\ge \frac 12\implies \frac 12|\na u|^2\ge \frac 12|\na w|^2-2|1-|u|^2| |\na u|^2.
\ee

Combining \eqref{bn3} and \eqref{bn4} with the upper bound \eqref{ar2} and the fact that $|u|\ge 1/2$ outside the bad discs, we find that
\be
\l{bn5}
G_\ve(u, \O)\ge \frac 12\int_{\O\setminus\cup_j B_j}|\na u|^2\ge \pi D \ln \frac 1\ve+2\pi\ln \frac 1{\bf m}-\widetilde C,
\ee
where $\widetilde C$ is a finite constant independent of small $\ve$. On the other hand, if $\delta\le \delta_3$, with $\delta_3\le\delta_2$ to be fixed later, \eqref{bg7} (with $\delta_3$ instead of $\delta_2$) holds, 
and thus, using \eqref{bn5}, we find that
\be
\l{bn6}
\ln \frac 1{\bf m}\le \frac D2\left(\frac 1{1-\delta_3}-1\right)\ln\frac 1\ve+\frac{\widetilde C+C_3}{2\pi}.
\ee

From \eqref{bn6}, we find that, for small $\ve$, \eqref{bm3} holds provided 
\bes
\frac D2\left(\frac 1{1-\delta_3}-1\right)<\alpha.\qedhere
\ees
\ep

When $\O$ is a general domain, we do not have \eqref{ar2} at our disposal anymore. We sketch below the adapted argument.

\bp[Sketch of proof of \eqref{bm3} in the general case] The inequality \eqref{bn3} still holds in the general case. However, the strategy for obtaining an analogue of \eqref{bn5} from \eqref{bn3} is different.  Consider some number $0<a<1$ to be fixed later. Define the (enlarged bad discs) as in Section \ref{as9}, but with $1/2$ replaced with $a$.  One can see that Lemma \ref{at1} still holds, possibly with some $N$ depending on $a$. Also, the analysis in Section \ref{bb1} holds, for $\delta\le \delta_0$, with $\delta_0$ possibly depending on $a$. So does Theorem \ref{bg1}. At this stage, using Theorem \ref{bg1} and Corollary \ref{be8} we conclude that, for small $\delta$ and two different $a$'s, the corresponding enlarged bad discs coincide, up to a multiplicative constant factor of their radii.  
Therefore, the estimate \eqref{bm3} does not depend on the specific value of $a$ we choose. 

\smallskip
We next explain how to choose $a$. We have the following substitute of \eqref{bn4}:
\be
\l{bn7}
|u|\ge a\implies |\na u|^2\ge a^2|\na w|^2.
\ee

From \eqref{bn3} (with the enlarged bad discs corresponding to $a$) and \eqref{bn7}, we find that
\be
\l{bn8}
G_\ve(u, \O)\ge \pi a^2 D\ln\frac 1\ve+2\pi a^2 \ln\frac 1{\bf m}-C(a).
\ee

In order to obtain \eqref{bm3} from \eqref{bn8} and \eqref{bg7}, it then suffices to choose $a$ and $\delta_3$ such that
\bes
\frac D{2 a^2}\left(\frac 1{1-\delta_3}-a^2\right)<\alpha.\qedhere
\ees
\ep

\br
\l{bn81}
For a different approach to the case of general (i.e., not necessarily strictly star-shaped) domains $\O$, see also Lemma \ref{bu8} and its applications in Section \ref{bu1}.
\er

The basic ingredient of the proof of \eqref{bm4} is the following simple result.
\bl
\l{bo1}
Let $0<\lambda<2\pi$. 
Let $g\in C^1(\partial\O ; \so)$. Then there exists some $r_0=r_0(\lambda, \O, g)$ such that the following holds. Let $0<r\le r_0$ and $x_0\in\p\O$. Consider a Lipschitz map $w:\p (B_r(x_0)\cap\O)\to\so$ such that $w=g$ on $B_r(x_0)\cap\p\O$ and $\deg w =-1$. Then
\be
\l{bo2}
\frac r2\int_{C_r(x_0)\cap \O}|w_\tau|^2\ge \lambda.
\ee

Similarly if $\deg w=d\in\Z$ and $0<\lambda< 2d^2\pi$.
\el
\bp
For small $r_0$ and with  finite constants $C_0$, $\widetilde C_0$, all depending only on $\O$, and for $r$, $x_0$ as above, the following hold: 
\begin{gather}
\l{bo3}
 \text{$C_r(x_0)\cap\O$ consists of a single arc of endpoints $a=a(r, x_0)$, $b=b(r,x_0)$},
 \\
 \l{bo4}
  {\mathscr H}^1(C_r(x_0))\le \pi r+ C_0 r^2,
  \\
  \l{bo5}
  \dist(a, b)\le \widetilde C_0 r
 \end{gather}
(in the last line, $\dist$ is the geodesic distance on $\p\O$).

\smallskip
Using: (i) \eqref{bo3}--\eqref{bo5}; (ii) the degree condition on $w$; (iii) the fact that $g$ is Lipschitz; (iv) the Cauchy-Schwarz inequality, we find successively, possibly after considering a smaller $r_0$:
\begin{gather*}
\int_{\p (B_r(x_0)\cap \O)}|w_\tau|\ge 2\pi,
\\
\int_{C_r(x_0)\cap\O}|w_\tau|\ge 2\pi-\int_{B_r(x_0)\cap \p\O}|w_\tau|\ge 2\pi-C(g)\widetilde C_0 r,
\\
(\pi r+C_0r^2)\int_{C_r(x_0)\cap\O}|w_\tau|^2\ge (2\pi-C(g)\widetilde C_0 r)^2,
\end{gather*}
and the last line implies \eqref{bo2} (since $\lambda<2\pi$), provided $r_0$ is sufficiently small.
\ep

In the proof of \eqref{bm4}, we will use Lemma \ref{bo1} in conjunction with the following lower   bound, which is a simple  consequence of the Cauchy-Schwarz inequality:
\be
\l{bo6}
\begin{split}
[0<r_1<r_2,\, x\in\R^2,\,  w: B_{r_2}(x)\setminus B_{r_1}(x)\to \so, \, w \text{ Lipschitz},\\
\deg (w, C_{r_1}(x))=d\in\Z]
\implies \frac 12\int_{ B_{r_2}(x)\setminus B_{r_1}(x)}|\na w|^2\ge \pi d^2\ln \frac {r_2}{r_1}.
\end{split}
\ee

\bp[Proof of \eqref{bm4}] Let $0<\alpha_1<\alpha$ be a constant to be fixed later. We take $\delta_3$ such that, when $0\le \delta\le \delta_3$, \eqref{bm3} holds for $\alpha_1$ instead of $\alpha$. Let $\ve$ be sufficiently small. Let $B_j=B_{C_1\ve}(x^j_\ve)$, $1\le j\le D$, be the enlarged bad discs. Assume, with no loss of generality, that the enlarged bad disc closest to $\p\O$ is $B_1$. If $\dist(x^1_\ve, \p\O)\ge \ve^\alpha$, then we are done. Otherwise, by choosing appropriate values of $\delta_3$,  $\alpha_1$, and of $\lambda$ in Lemma \ref{bo1}, we will obtain a contradiction for small $\ve$.  For this purpose, we first use \eqref{bo6} and obtain the following inequalities
\begin{gather}
\l{bo7}
\frac 12\int_{B_{\ve^{\alpha_1}}(x^j_\ve)\setminus B_j}|\na w|^2\ge \pi \ln \frac{\ve^{\alpha_1}}{C_1\ve},\ \fo j\ge 2,
\\
\l{bo8}
\frac 12\int_{B_{\ve^\alpha}(x^1_\ve)\setminus B_1}|\na w|^2\ge \pi \ln \frac{\ve^\alpha}{C_1\ve}.
\end{gather}

We next use Lemma \ref{bo1} and obtain, for small $\ve$ and with $x_0$ the nearest point projection of $x^1_\ve$ on $\p\O$, the bound
\be
\l{bo9}
\frac 12 \int_{B_{\ve^{\alpha_1}/2}(x_0)\setminus B_{2\ve^\alpha} (x_0)}|\na w|^2\ge \lambda\ln \frac{\ve^{\alpha_1}/2}{2\ve^\alpha}.
\ee

We next note that, for small $\ve$, the integration domains in \eqref{bo7}--\eqref{bo9} are mutually disjoint. 
Combining this observation with \eqref{bo7}--\eqref{bo9} and using the fact that $w$ is smooth and fixed outside $\O$, we obtain the lower bound
\be
\l{bo10}
\frac 12\int_{\O\setminus\cup_j B_j} |\na w|^2\ge [\pi (1-\alpha_1) D+(\lambda-\pi) (\alpha-\alpha_1)]\ln\frac 1\ve-C,
\ee
where $C$ is a finite numerical constant (depending on $C_1$ and on the extension of $u$ outside $\O$). 

\smallskip
On the other hand, we know from the proof of \eqref{bm3} in a general domain that, possibly after  modifying the construction of the enlarged bad discs as explained there, we have, for a given $0<a<1$ and sufficiently small $\ve$, 
\be
\l{bo11}
G_\ve(u,\O)\ge a^2\frac 12 \int_{\O\setminus\cup_j B_j} |\na w|^2.
\ee

The estimates \eqref{bo10} and \eqref{bo11} contradict, for small $\ve$, the upper bound \eqref{bg7} (with $\delta_3$ instead of $\delta_2$), provided we have
\be
\l{bo12}
a^2[\pi (1-\alpha_1) D+(\lambda-\pi) (\alpha-\alpha_1)]>\frac{\pi}{1-\delta_3}D.
\ee

We finally note that \eqref{bo12} holds for 
any constant $\lambda>\pi$, provided   we let $\alpha_1$ and $\delta_3$ sufficiently small and $a$ sufficiently close to $1$.
\ep

\br
\l{bp1} 
For the standard Ginzburg-Landau energy and in a strictly star-shaped domain $\O$, the method of proof of \eqref{bm3}--\eqref{bm4} allows us to recover a \enquote{repelling effect} initially established in \cite{bbh}:  for small $\ve$,  the mutual distances between the bad discs and the distances from the bad discs to $\p\O$ is above some positive constant. 

Indeed, combining \eqref{bn3}  with the upper bound \eqref{aq2} and the inequalities \eqref{bn4} and \eqref{ar2} (for the latter one, we rely on the fact that $\O$ is strictly star-shaped), we see that 
\be
\l{bp2}
\liminf_{\ve\to 0}\min_{j\neq k}|x^j_\ve-x^k_\ve|\ge C>0
\ee
for some constant $C$ depending on $\O$ and $g$.

It remains to prove that the bad discs cannot get close to the boundary. Argue by contradiction and assume, e.g., that, possibly up to a subsequence,   ${\bf m}:=\dist (x^1_{\ve}, \p\O)\to 0$. With  $C$ as above, we may repeat the proof of \eqref{bo7}--\eqref{bo9}, with $\ve^{\alpha_1}$ replaced with $C/4$, and find ({\it via} Lemma \ref{bo1} and \eqref{bo6}) that
\be
\l{bp3}
\frac 12\int_{\cup_j (B_{C/4}(x^j_\ve)\setminus B_{C_1\ve}(x^j_\ve))}|\na w|^2\ge \pi D \ln\frac 1\ve+(\lambda-\pi)\ln\frac 1{\bf m}-C.
\ee

Once $\lambda>\pi$ is fixed, the inequality \eqref{bp3} contradicts, for small $m$, the upper bound \eqref{aq2}.
\er

\section{Toy minimization problems on an annulus}
\l{bq1}

There exists a natural construction of competitors for the minimization problem $\min_{H^1_g(\O ; \so)}E_\ve$ when $\deg g=-D<0$. More specifically, set, for $0<R_1<R_2<\infty$,
\bes
A_{R_1, R_2}:={\overline B}_{R_2}(0)\setminus B_{R_1}(0).
\ees

 Consider the class
 \be
 \l{bq2}
 \ba
 {\mathscr H}_{R_1, R_2, C}:=\{v\in H^1(A_{R_1, R_2} ; \so); & \, |v_\theta|\le C\text{ on }C_{R_1}(0)\text{ and }C_{R_2}(0),
 \\
&\deg (v, C_{R_1}(0))=\deg (v, C_{R_2}(0))=-1\}
\ea
 \ee
 and the minimization problem
 \be
\l{bq3}
I_{R_1, R_2, C}:=\min\{ E_0(v); \,  v\in  {\mathscr H}_{R_1, R_2, C}\}.
\ee
 
 The  class  ${\mathscr H}_{R_1, R_2, C}$ is non-empty if $C\ge 1$, since for such $C$ it contains the map
 \bes
 u_0:A_{R_1, R_2}\to\so,\ u_0(z):=\frac{\overline z}{|z|},\ \fo z\in A_{R_1, R_2}.
 \ees 
{\it From now on, we assume that $C\ge 1$.}
 
 \smallskip
 It is straightforward that there exists a minimizer $u_{R_1, R_2, C}$ in \eqref{bq3}. In the special case  where $\delta=0$, \eqref{bq3} is equivalent to the minimization of the standard Dirichlet integral $G_0(v)$, and the above $u_0$ is a minimizer. We conjecture that even when $\delta>0$, a minimizer likewise could be a $0$-homogeneous map (thus a function depending only on $\theta$ and independent of $r=|x|$, $R_1$, $R_2$, and $C$), but we are not aware of a proof of this fact.

\smallskip
Starting from a minimizer $w:=u_{\ve, R_2, C}$ in \eqref{bq3}, where $0<\ve<R_2$ and $R_2$ is sufficiently small (depending on $\O$) and fixed, we construct  a competitor $u$ in $H^1_g(\O ; \so)$ as follows. Consider $D$ disjoint closed balls $\overline B_{R_2}(x_j)$, $j=1,\ldots, D$, contained in $\O$. Let $v$ be the restriction of $w$ to $C_\ve(0)$. We first define $u$ in each $\overline B_{R_2}(x_j)$ by setting
\be
\l{bq4}
u(x)=u(r e^{\im\theta}+x_j):=\begin{cases}
w(x-x_j),&\text{if }\ve\le  r\le R_2
\\
(r/\ve) v(\ve e^{\im\theta}), &\text{if } r\le \ve
\end{cases}.
\ee

We next extend $u$ to $\O$ by considering, in $\O\setminus \cup_j B_{R_2}(x_j)$, an $\so$-valued map, still denoted $u$, agreeing with the above map on each $C_{R_2}(x_j)$ and taking the value $g$ on $\p\O$. It is straightforwards that this is possible such that, in addition,
\be
\l{bq4a}
|\na u|\le C_1(C, \O, g)\ \text{in }\overline\O \setminus \cup_j B_{R_2}(x_j).
\ee

Fixing the value of $C$ and using  $u$ as a competitor, we obtain the upper bound
\be
\l{bq5}
m_\ve:=\min \{ E_\ve(u);\, u\in H^1_g(\O ; \C)\}\le D I_{\ve, R_2,C}+C_2(\O, g).
\ee

A remarkable result of Bethuel, Brezis, and H\'elein \cite{bbh} asserts that, when $\delta=0$ and $\O$ is strictly star-shaped, this construction provides the correct asymptotics of $m_\ve$ up to a bounded error, that is, 
\be
\l{bq6}
m_\ve= D I_{\ve, R_2,C}+O(1)=\pi D \ln \frac 1\ve+O(1)\ \text{as }\ve \to 0.
\ee

Among other ingredients of the proof of \eqref{bq6} in \cite{bbh}, there is the exact formula for $I_{\ve, R_2,C}$. Although we are not aware of such a formula when $\delta\neq 0$, we will establish, in the next two sections, analogues of \eqref{bq6} valid for small $\delta$. 

\smallskip
In the current section, we investigate some basic properties of $I_{R_1, R_2, C}$,
that we collect in the following simple result.
\bl
\l{br0}
Let $C, C'\ge 1$. Let $0<R_1\le R_2\le R_3\le R_4<\infty$. Then the following properties hold, with $C_j$ constants depending only on the variables indicated below.
\begin{gather}
\l{bra}
(1-\delta)\pi \ln \frac{R_2}{R_1}-(1-\delta)\pi\le I_{R_1, R_2, C}\le \pi \ln \frac{R_2}{R_1}.
\\
\l{br1}
I_{tR_1, tR_2, C}=I_{R_1, R_2, C},\ \fo t>0.
\\
\l{br2}
\text{If $R_3/R_2\ge 2$, then }I_{R_1, R_4, C}\le I_{R_1, R_2, C}+I_{R_3, R_4, C}+C_1(C,\delta)\ln \frac{R_3}{R_2}+C_2(C,\delta).
\\
\l{br3}
\text{If $t\ge 2$, then }I_{R_1, t R_2, C}\le I_{R_1, R_2, C}+C_1(C,\delta)\ln t+C_2(C,\delta).
\\
\l{br4}
\text{If  $t\ge 2$, then }I_{R_1, t R_2, C'}\le I_{R_1, R_2, C}+C_3(C,\delta)\ln t+C_4(C,\delta).
\\
\l{br5}
\ba
&
\text{There exists a constant $(1-\delta)\pi\le C_\delta\le \pi$ }\text{such that}
\\
&\lim_{t\to\infty}\frac{I_{R_1, tR_1, C}}{\ln t}=C_\delta,\ \fo R_1>0,\ \fo C\ge 1.
\ea
\\
\l{br6}
\text{If $\delta>0$, then $C_\delta<\pi$}.
\\
\l{cg1}
\text{If $\delta>0$, then $C_\delta>(1-\delta)\pi$}.
\end{gather}
\el
\noindent{\it Proof of \eqref{bra}.} For the left-hand side, we use \eqref{af5} combined with the fact that $u_0$ is a minimizer of $G_0$ in the class ${\mathscr H}_{R_1, R_2, C}$. For the right-hand side, we consider, as in the proof of Lemma \ref{aq1}, the competitor $u_0$.

\smallskip
\noindent
{\it Proof of \eqref{br1}}. This identity follows from the fact that $E_0$ is invariant under homotheties, and so is the condition $|v_\theta|\le C$.

\smallskip
\noindent
{\it Proof of \eqref{br2}.}  
Let $u_2$, respectively $u_3$, be a minimizer  in ${\mathscr H}_{R_1, R_2, C}$, respectively in ${\mathscr H}_{R_3, R_4, C}$. Let $v_2$, respectively $v_3$, be the trace of $u_2$ on $C_{R_2}(0)$, respectively of $u_3$ on $C_{R_3}(0)$.  

\smallskip
Since we have $|(v_j)_\theta|\le C$, $j=2, 3$, and $\deg v_2=\deg v_3=-1$, we may write $v_j(R_j e^{\im\theta})=\exp\, (\im(-\theta+\psi_j(\theta)))$, with 
\be
\l{bs2}
|(\psi_j)_\theta-1|\le C\ \text{and}\ |\psi_j|\le (1+C)\pi. 
\ee

We next \enquote{interpolate} between $v_2$ and $v_3$ by setting
\bes
u((1-\sigma)R_2+\sigma R_3) e^{\im\theta}):=\exp \, (\im(-\theta+(1-\sigma)\psi_2(\theta)+\sigma \psi_3(\theta))),\ \fo 0\le \sigma\le 1.
\ees

Clearly, $u$ agrees with $v_2$ on $C_{R_2}(0)$ and with  $v_3$ on $C_{R_3}(0)$. On the other hand, one has (from \eqref{bs2})
\be
\l{bs1}
|u_r|\le \frac{2(1+C)\pi}{R_3-R_2} \ \text{and} \ |u_\theta|\le C.
\ee

Considering the competitor $\begin{cases}
u_2,&\text{in }A_{R_1, R_2}
\\
u,&\text{in }A_{R_2, R_3}
\\
u_3,&\text{in }A_{R_3,R_4}
\end{cases}$ in the class ${\mathscr H}_{R_1,R_4,C}$ and using, in $A_{R_2,R_3}$, \eqref{af4} and\eqref{bs1}, we obtain that \eqref{br2} holds with $C_1(C,\delta):=(1+\delta) \pi C^2$ and $C_2(C,\delta):=6(1+C)^2\pi^3+2\pi(1+\delta)$.

\smallskip
\noindent
{\it Proof of \eqref{br3}.} This is a special case of \eqref{br2}, with $R_3:=t R_2$ and $R_4=R_3$.

\smallskip
\noindent
{\it Proof of \eqref{br4}.} We essentially repeat the proof of \eqref{br2}. Given a minimizer $u_1$ in ${\mathscr H}_{R_1, R_2, C}$ and letting $v_2$ be the restriction of $u_1$ to $C_{R_2}(0)$, we interpolate, in $\overline B_{tR_2}(0)\setminus B_{R_2}(0)$, between $v_2$ and $\theta\mapsto e^{-\im\theta}$ in order to construct a competitor in ${\mathscr H}_{R_1, t R_2, C'}$.

\smallskip
\noindent
{\it Proof of \eqref{br5}.} By \eqref{br1}, it suffices to investigate the case where $R_1=1$. Set 
\be
\l{bs3}
C_\delta:=\liminf_{t\to\infty}\frac{I_{1,t, 1}}{\ln t}.
\ee

We will prove that \eqref{br5} holds for this  $C_\delta$. To start with, we note that, 
by \eqref{bra}, we have $(1-\delta)\pi\le C_\delta\le \pi$.

\smallskip
We next prove that
\be
\l{bs5}
\limsup_{t\to\infty}\frac{I_{1,t, 1}}{\ln t}\le \liminf_{t\to\infty}\frac{I_{1,t, 1}}{\ln t}.
\ee

Let  $\ve>0$ and let $M>1$ to be fixed later in function of $\ve$. Let $t_0\ge M$ be such that
\be
\l{bs6}
\frac{I_{1, t_0, 1}}{\ln t_0}< C_\delta+\frac\ve 2. 
\ee

For   $t\ge 2 t_0$, let 
\be
\l{bs7}
k:=\left\lfloor \frac{\ln t}{\ln (2t_0)}\right\rfloor\ge 1,
\ee
so that 
\be
\l{bs7a}
(2t_0)^k\le t< (2t_0)^{k+1}. 
\ee

By applying repeatedly \eqref{br1}--\eqref{br3}, we obtain, by a straightforward induction on $k$, that
\be
\l{bs8}
I_{1, t, 1}\le  k I_{1, t_0,1}+(k-1)[C_1(1,\delta)\ln 2+C_2(1,\delta)] + C_1(1,\delta)\ln (2t_0)+C_2(1,\delta).
\ee

To illustrate the proof of \eqref{bs8}, we detail, for example, the case where $k=2$. Then, using successively \eqref{br2}, \eqref{bs7a}, \eqref{br2}, \eqref{br1}, we find that
\bes
\ba
I_{1, t,1}\le &I_{1, 2t_0^2, 1}+C_1(1,\delta)\ln \frac t{2t_0^2}+C_2(1,\delta)
\\
\le & I_{1, 2t_0^2, 1}+C_1(1,\delta)\ln ({\color{red}4}t_0)+C_2(1,\delta)
\\
\le & I_{1,t_0,1}+I_{2t_0,2t_0^2,1}+C_1(1,\delta)\ln 2+C_2(1,\delta)+C_1(1,\delta)\ln ({\color{red}4}t_0)+C_2(1,\delta)
\\
= & I_{1,t_0,1}+I_{1,t_0,1}+C_1(1,\delta)\ln 2+C_2(1,\delta)+C_1(1,\delta)\ln ({\color{red}4}t_0)+C_2(1,\delta),
\ea
\ees
and the last line coincides with the right-hand side of \eqref{bs8} with $k=2$.

\smallskip
Dividing \eqref{bs8} by $k\ln (2t_0)$, letting $k\to\infty$, and taking \eqref{bs7} into account, we find that
\be
\l{bs9}
\limsup_{t\to\infty}\frac {I_{1,t,1}}{\ln t}\le \frac{I_{1,t_0,1}}{\ln t_0}\frac{\ln (2t_0)}{\ln t_0}+\frac {C_1(1,\delta)\ln 2+C_2(1,\delta)}{\ln (2t_0)}.
\ee

We next note that, when  $M>1$ is sufficiently large (depending on $\ve$), $t_0\ge M$, and  \eqref{bs6} holds, the right-hand side of \eqref{bs9} is $<C_\delta+\ve$. Therefore, \eqref{bs5} holds.

\smallskip
To complete the proof of \eqref{br5}, we note the straightforward inequality $I_{1,t,C}\le I_{1,t,1}$. Combining this with \eqref{br4}, we find that, when $t\ge 2$, we have
\be
\l{bs10}
I_{1,t/2,1}+C_3(C,\delta)\ln 2+C_4(C,\delta)\le I_{1,t,C}\le I_{1,t,1}.
\ee

We obtain \eqref{bs5} for an arbitrary constant $C\ge 1$ {\it via} \eqref{bs3}, \eqref{bs5}, and \eqref{bs10}.

\smallskip
\noindent
{\it Proof of \eqref{br6}.} We consider a $C^1$ map $f:\so\to\so$, of degree $-1$, and the competitor $u(r e^{\im\theta}):=f(e^{\im\theta})$, $\fo r>0$. Clearly, for some $C$ depending on $f$, we have $u\in {\mathscr H}_{1,t,C}$. On the other hand, if we write $f(e^{\im\theta})=\exp (\im (-\theta+\psi(\theta)))$, with $\psi$ of class $C^1$ and $2\pi$-periodic, we have
\begin{gather*}
(\div u)(r e^{\im\theta})=\frac 1r (\psi'(\theta)-1) \cos (\psi(\theta)-2\theta),\\
 (\curl u)(r e^{\im\theta})=\frac 1r (\psi'(\theta)-1) \sin (\psi(\theta)-2\theta),
\end{gather*}
and thus
\be
\l{bs11}
\ba
\frac{E_0(u ; A_{1,t})}{\ln t}&=\int_0^{2\pi}\frac{K_1\cos^2(\psi(\theta)-2\theta)+K_3 \sin^2(\psi(\theta)-2\theta)}2(\psi'(\theta)-1)^2\, d\theta
\\
&:= {\mathscr I}(\psi).
\ea
\ee

From \eqref{bs11} and \eqref{br5}, we find that
\be
\l{bs12}
\ba
C_\delta\le &\inf\{ {\mathscr I}(\psi);\, \psi\in C^1([0,2\pi]; \R),\, \psi(0)=\psi(2\pi)\}
\\
=& \min\{ {\mathscr I}(\psi);\, \psi\in H^1([0,2\pi]; \R),\, \psi(0)=\psi(2\pi)\}.
\ea
\ee

By the direct method in the calculus of variations, the $\min$ in the second line of \eqref{bs12} is achieved, and any minimizer satisfies 
\be
\l{bs13}
\ba
F_\psi(\psi(\theta), \theta)&(\psi'(\theta)-1)^2-2[F(\psi(\theta),\theta)(\psi'(\theta)-1)]_\theta=0,
\\
&\text{where }F(\psi,\theta):=\frac {K_1}2 \cos^2(\psi-2\theta)+\frac {K_3}2 \sin^2 (\psi-2\theta).
\ea
\ee

Let us note that 
\be
\l{bs14}
F_\psi(\psi, \theta)=-2\delta \sin(\psi-2\theta)\cos (\psi-2\theta).
\ee

Using \eqref{bs14}, we find that, when $\psi=0$, the left-hand side of the first line of \eqref{bs13} equals $ -6\delta \sin (2\theta) \cos (2\theta)$, and thus, when $\delta\neq  0$, $\psi=0$ is not a minimizer of ${\mathscr I}$. Combining this with the fact that, when $\psi=0$, we have $E_0(u)=\pi \ln t$ (see the proof of \eqref{bra}) and with \eqref{bs12}, we obtain that $C_\delta<\pi$ when $\delta>0$.

\smallskip
\noindent
{\it Proof of \eqref{cg1}.} Let $R_1:=1$, $R_2:=t>1$. Any competitor $v$ in \eqref{bq3} is of the form 
\bes
v=\exp\, (-\im (\theta+\psi)), \ \text{with }\psi\in H^1(A_{1, t}). 
\ees

For $v$ as above, we find,  with $\O:=A_{1,t}$, using the fact that $\Jac v=0$ in $\O$ (since $v$ is $\so$-valued):
\begin{gather*}
\int_\O [(\div v)^2+(\curl v)^2]=\int_\O |\na v|^2+2\int_\O \Jac v=\int_\O |\na v|^2,
\\
\ba
\int_\O |\na v|^2=&\int_\O |\na (\theta+\psi)|^2=\int_\O |\na \theta|^2+\int_\O |\na \psi|^2+ 2\int_\O \frac 1 r \psi_\tau
\\
=&\int_\O |\na\theta|^2+\int_\O |\na\psi|^2=2\pi \ln t +\int_\O |\na\psi|^2,
\ea
\\
\ba
E_0(v)=&\frac{1-\delta}2\int_\O [(\div v)^2+(\curl v)^2]+\delta\int_\O (\div v)^2
\\
=& (1-\delta)\pi \ln t+\frac{1-\delta}2 \int_\O |\na \psi|^2+\delta\int_\O \left[\frac 1r \cos (\psi-2\theta)+\begin{pmatrix}\sin (\theta-\psi)\\
\cos (\theta-\psi)\end{pmatrix}\cdot \na\psi\right]^2.
\ea
\end{gather*}

Assume, by contradiction, that $C_\delta=(1-\delta)\pi$. Then there exist $t_j\to\infty$ and $\psi_j\in H^1(A_{1,t_j})$ such that, with $\O_j:=A_{1, t_j}$, we have
\be
\l{cg2}
\begin{split}
\int_{\O_j}|\na\psi_j|^2+\int_{\O_j}\left[\frac 1r \cos (\psi_j-2\theta)+\begin{pmatrix}\sin (\theta-\psi_j)\\
\cos (\theta-\psi_j)\end{pmatrix}\cdot \na\psi_j\right]^2=o(\ln t_j)
\\
\text{as }j\to \infty.
\end{split}
\ee

By \eqref{cg2} we obtain (using Cauchy-Schwarz on the third line) that
\begin{gather}
\l{cg3}
\int_{\O_j}\left[\frac 1r \psi_{j,\theta}\right]^2\le \int_{\O_j} |\na \psi_j|^2=o(\ln t_j)
\  \text{as }j\to \infty,
\\
\notag
\int_{\O_j}\left[\begin{pmatrix}\sin (\theta-\psi_j)\\
\cos (\theta-\psi_j)\end{pmatrix}\cdot \na\psi_j\right]^2=o(\ln t_j)
\  \text{as }j\to \infty,
\\
\notag
\int_{\O_j}\frac 1r \cos (\psi_j-2\theta)\times \begin{pmatrix}\sin (\theta-\psi_j)\\
\cos (\theta-\psi_j)\end{pmatrix}\cdot \na\psi_j=o(\ln t_j)
\  \text{as }j\to \infty,
\\
\l{cg5}
\int_{\O_j}\left[\frac 1r \cos (\psi_j-2\theta)\right]^2=o(\ln t_j)
\  \text{as }j\to \infty.
\end{gather}

Combining \eqref{cg3} and \eqref{cg5} with a mean value argument, we find that there exist radii $1<r_j<t_j$ such that 
\bes
\frac 1{r_j}\int_{C_{r_j}(0)}\left\{(\psi_{j, \theta})^2+[\cos (\psi_j-2\theta)]^2\right\}\to 0\ \text{as }j\to \infty.
\ees

Therefore, if we set 
\bes
g_j(e^{\im\theta}):=\psi_j(r_j e^{\im\theta}),\ \fo j,\,  \fo \theta,
\ees
we have 
\be
\l{cg6}
\int_{\so} \left\{(g_{j, \theta})^2+[\cos (g_j-2\theta)]^2\right\}\to 0\ \text{as }j\to \infty.
\ee

After subtracting a suitable multiple of $2\pi$ from $g$, we may assume that $0\le g_j(e^{\im\theta_j})\le 2\pi$ for some $\theta_j$, and then \eqref{cg6} implies that, possibly up to a subsequence, there exists some constant $C$ such that $g_j \to C$ uniformly. We obtain from \eqref{cg6} that
\bes
\int_{\so} [\cos (C-2\theta)]^2=0,
\ees
which is impossible. This contradiction completes the proof.\hfill\qedsymbol{}

\medskip
While the above considerations will suffice to yield the correct asymptotics of the minimal energy $m_\ve$ when $D=1$, for higher degrees we rely on the study of a \enquote{cousin} of ${\mathscr H}_{R_1,R_2,C}$. More specifically, when $C\ge 1$, we consider the class
\be
 \l{bt1}
 \ba
 \tilde{\mathscr H}_{R_1, R_2,  C}:=\{v\in H^1(A_{R_1, R_2} ; \so); & \, \int_{C_{R_j}(0)}|v_\theta|^2\le 2\pi R_j C^2, \, j=1,2,
  \\
&\deg (v, C_{R_1}(0))=\deg (v, C_{R_2}(0))=-1\}
\ea
 \ee
 and the minimization problem
 \be
\l{bt2}
\tilde I_{R_1, R_2, C}:=\min\{ E_0(v); \,  v\in  \tilde {\mathscr H}_{R_1, R_2, C}\}.
\ee
 
The following result is a straightforward version of (part of) Lemma \ref{br0}, and its proof is omitted. 

\bl
\l{bt0}
Let $C\ge 1$. Let $0<R_1\le R_2\le R_3\le R_4<\infty$. Then the following properties hold, with $C_j$ constants depending only on the variables indicated below.
\begin{gather}
\l{bt3}
{\mathscr H}_{R_1, R_2, C}\subset \tilde{\mathscr H}_{R_1,R_2,C}, \text{ and therefore}\ \tilde I_{R_1,R_2, C}\le I_{R_1,R_2,C}.
\\
\l{bt4}
\tilde I_{tR_1, tR_2,  C}=\tilde I_{R_1, R_2,  C},\ \fo t>0.
\\
\l{bt5}
\text{If $R_3/R_2\ge 2$, then }I_{R_1, R_4, C}\le \tilde I_{R_1, R_2, C}+\tilde I_{R_3, R_4, C}+C_1(C,\delta)\ln \frac{R_3}{R_2}+C_2(C,\delta).
\\
\l{bt6}
\text{With $C_\delta$ as in \eqref{br5}},\ 
\lim_{t\to\infty}\frac{\tilde I_{R_1, tR_1, C}}{\ln t}=C_\delta,\ \fo R_1>0,\ \fo  C\ge 1.
\end{gather}
\el

\section{Small \texorpdfstring{$\delta$}{delta} analysis. More on the location of bad discs. Asymptotic expansion of the energy}
\l{bu1}

We derive here a number of consequences of the results established in Sections  \ref{bg0}--\ref{bq1}; in particular, we improve the conclusion of \eqref{bm3} in  Theorem \ref{bm2}.

\smallskip
In what follows, we consider some integer $D\ge 1$. Given a domain $\O$ and a boundary condition $g:\p\O\to\so$ of degree $-D$, we let 
\be
\l{bu0}
m_\ve=m_{\ve, \O, g}:=\min\{ E_\ve(u);\, u\in H^1_g(\O; \C)\}.
\ee

Let $\delta_2=\delta_2(D)$ be as is defined in Theorem \ref{bg1}, and let  $0\le \delta\le \delta_2$. By Theorem \ref{bg1}, for small $\ve$, a map $u=u_\ve$ achieving $m_\ve$ has exactly $D$ enlarged bad discs of centers $x^1_\ve,\ldots, x^D_\ve$.

\smallskip
We start with an easy result.
\bt
\l{bu2}
Let $D=1$. 
 Let $\delta\le \delta_2(1)$. Then, for any  $C\ge 1$, we have
\be
\l{bu3}
m_\ve= I_{\ve, 1, C} +O(1)\ \text{as }\ve \to 0.
\ee

In particular, we have
\be
\l{bu3a}
m_\ve=C_\delta \ln \frac 1\ve+o\left(\ln \frac 1\ve\right)\ \text{as }\ve \to 0.
\ee
\et

In the above, $O(1)$ stands for a quantity such that $|O(1)|\le C(\delta, \O, g)<\infty$ as $\ve \to 0$.

\smallskip
We continue with a significant improvement of \eqref{bm3}.

\bt
\l{bu8}
Let $D\ge 2$. 
Let $0<\alpha_0<1$. Assume that $\delta_3=\delta_3(D)<\min (\delta_2(D), 2/(D+2))$ is such that:  if $0\le\delta\le\delta_3$, there exists some $0<\alpha_0=\alpha_0(\delta, \O, g)<1$ with the property that, when  $\ve$ is small (smallness depending on $\delta$),  the centers $x^j_\ve$, $j=1,\ldots, D$, of the enlarged bad discs satisfy
\be
\l{bu99}
 {\bf m}:=\frac 12\min_{j\neq k}|x^j_\ve-x^k_\ve|\ge \ve^{\alpha_0}.
 \ee
 
 Then, for every $0<\alpha<1$ and for $\delta\le \delta_3$ as above, we have, for small $\ve$ (smallness depending on $\alpha$ and $\delta$),
 \be
 \l{bu10}
 {\bf m}\ge \ve^\alpha.
 \ee
 
 In particular, there exists some $\delta_3>0$ such that \eqref{bu10} holds for each $0<\alpha<1$ provided $\ve$ is sufficiently small (smallness depending on $\alpha$ and $\delta$).
\et

The heart of the matter consists of establishing the first part of Theorem \ref{bu8}; the second part of Theorem \ref{bu8} follows from the first part combined with \eqref{bm3}. 

\smallskip
Note that, while in Theorem \ref{bm2} $\delta_3$ depends on both $D$ and $\alpha$, the conclusion of the second part of Theorem \ref{bu8} is that $\delta_3$ can be chosen to depend only on $D$.

\smallskip
An easy consequence of Theorem \ref{bu8} is the following.
\bt
\l{bu5} Let $D\ge 2$.
Let $\delta\le \delta_3(D)$, with $\delta_3(D)$ as in Theorem \ref{bu8}.  Then 
\be
\l{bu6}
m_\ve= D C_\delta \ln\frac 1\ve+o\left(\ln \frac 1\ve\right)\ \text{as }\ve \to 0.
\ee
\et

Our next result complements Theorems \ref{bu2} and \ref{bu5}.
\bt
\l{bw1}
Let $D\ge 1$. If $D=1$, let $\delta\le \delta_2(1)$. If $D\ge 2$, let $\delta\le \delta_3(D)$.
Let $u_\ve$ achieve $m_\ve$. If, up to a subsequence, $x^j_\ve\to a_j\in\overline\O$, $j=1,\ldots, D$, then 
\be
\l{bu7}
\frac{e_\ve(u_\ve)}{\ln (1/\ve)}\rightharpoonup C_\delta\sum_j\delta_{a_j}\ \ \text{$\ast$-weakly in ${\mathscr M}(\overline\O)$}.
\ee
\et

\noindent
(Recall that 
\bes
e_\ve(u)=\d\frac{K_1}2 (\div u)^2+\frac{K_3} 2(\curl u)^2+\frac 1{4\ve^2}(1-|u|^2)^2
\ees
is the energy density.)

\smallskip
A basic tool used in the proofs of the above results is the following substitute of  \eqref{bn7}.
\bl
\l{bu8bis} Let $1/2\le a<1$ and $C=C(a)<\infty$ be such that, for the corresponding en\-larged bad discs, we have  $|u|\ge a$ in 
 $\omega:=\O\setminus\cup_j B_{C\ve}(x^j_\ve)$. Let $w:=u/|u|$ in $\omega$. Then
\be
\l{bu9}
E_0(u, \omega)\ge  E_0(w, \omega)-C_1(\delta, \O,g) \frac 1{\ve^2}\int_\O (1-|u|^2)^2-C_2(\delta, a, \O, g).
\ee
\el
\bp[Proof of Lemma \ref{bu8bis}] If $z=(z_1,z_2)\in \R^2\sim\C$, we set $z^\perp:=(-z_2, z_1)$.

\smallskip
Let $\rho:=|u|$ so that $u=\rho w$ in $\omega$ and
\begin{gather}
\l{bv1}
(\div u)^2=(\rho\div w+\na \rho\cdot w)^2\ge \rho^2(\div w)^2+\na (\rho^2-1)\cdot ((\div w) w)
\\
\l{bv2}
(\curl u)^2=(\rho\curl w-\na \rho\cdot w^\perp)^2\ge \rho^2(\curl w)^2-\na (\rho^2-1)\cdot ((\curl w) w^\perp).
\end{gather}

We integrate \eqref{bv1} over $\omega$, using an integration by parts for the last term. We proceed similarly for \eqref{bv2}. Combining the two results, we find that
\be
\l{bv3}
\ba
E_0(u, \omega)\ge & E_0(w,\omega)-\frac{K_1}{2}\int_\omega (1-\rho^2)(\div w)^2-\frac{K_3}2\int_\O (1-\rho^2)(\curl w)^2
\\
&+\frac{K_1}2 \int_{\p\omega} (\rho^2-1)(\div w)w\cdot\nu-\frac{K_3}2 \int_{\p\omega} (\rho^2-1) (\curl w)w^\perp\cdot \nu\\
&-\frac {K_1}2\int_\omega (\rho^2-1) (\div w)^2-\frac{K_3}2\int_\omega (\rho^2-1) (\curl w)^2
\\
&-\frac {K_1}2\int_\omega (\rho^2-1) w\cdot \na (\div w)+\frac {K_3}2\int_\omega (\rho^2-1) w^\perp\cdot \na (\curl w).
\ea
\ee

Using: (i) Corollary \ref{aw9} ; (ii) the fact that $|u|\ge 1/2$ in $\overline\omega$; (iii) the fact that $\rho=1$ on $\p\O$, we find that  the second line in \eqref{bv3} is $\ge -C_3$, where $C_3=C_3(\delta, a, \deg g)$. Using, for the other integrals in \eqref{bv3}, the fact that $|u|\ge 1/2$ in $\omega$, we find that
\be
\l{bv4}
\ba
E_0(u,\omega)\ge & E_0(w, \omega)-C_3-C_4\int_\omega |1-|u|^2|\times (|\na u|^2+|D^2 u|)
\\
\ge &
E_0(w, \omega)-C_3-C_4\int_\O |1-|u|^2|\times (|\na u|^2+|D^2 u|),
\ea
\ee
where $C_4$ is a  universal constant.

\smallskip
It remains to dominate the last integral in \eqref{bv4}. Using: (i) Cauchy-Schwarz; (ii) formula \eqref{as2} (except for the final inequality in \eqref{as2}, which requires that $\O$ is strictly star-shaped),  we find that
\be
\l{bv5}
\int_\O |1-|u|^2|\times (|\na u|^2+ |D^2 u|)\le C_5 (1+\ve^{-2}\vertii{1-|u|^2}_2),
\ee
where $C_5=C_5(\delta, \O, g)$. We obtain \eqref{bu9} from \eqref{bv4} and \eqref{bv5}.
\ep

\bp[Proof of Theorem \ref{bu2}] {\it Proof  when $\O$ is strictly  star-shaped and $C=1$.} Let $\ve$ be sufficiently small and let  $B_{C_1\ve}(x^1_\ve)$ be the enlarged bad disc corresponding to $u=u_\ve$. By choosing if needed a larger (but fixed) $C_1$, we may assume that, $\O\subset B_{C_1/4}(0)$ and thus $\O\subset B_{C_1/2}(x^1_\ve)$. Extend $u$ to $\R^2$ as explained in the proof of Theorem \ref{bg1}  (after formula \eqref{bk1}). Assume, to simplify the forthcoming formulas, that $x^1_\ve=0$. Using: (i) estimate \eqref{as6}; (ii) the competitor $w:=u/|u|$  in the minimization problem  \eqref{bq3} in $A_{C_1\ve, C_1/2}$ (where $C$ defining the class ${\mathscr H}_{C_1\ve, C_2/2, C}$ is a sufficiently large fixed constant depending on $g$ {\it via} \eqref{aqb} and the extension $G$); (iii) the upper bound \eqref{as6}; (iv) Lemma \ref{bu8}; (v) \eqref{br1}; (vi) \eqref{br4}, we find that
\be
\l{bv6}
m_\ve\ge I_{\ve, 1/2, C}+O(1)\ge I_{\ve, 1, 1}+O(1).
\ee
On the other hand, by combining \eqref{bv6} with \eqref{bq5}, \eqref{br1}, and \eqref{br3}, we find that
\be
\l{bv7}
m_\ve\le I_{\ve, 1, 1}+O(1).
\ee

We complete the proof via \eqref{bv6}  and \eqref{bv7}.

\smallskip
\noindent
{\it Boundedness of the potential term in a general domain.} This follows from a principle devised by Del Pino and Felmer \cite{delpino_felmer}. Assume for simplicity that $0\in\O$. Let $v_\ve(x):=u_{\ve}(2x)$, $x\in\R^2$ (where $u_\ve$ has been extended to $\R^2$ as above). Let $B_R(0)$ be a large fixed ball containing $\overline \O$. By the first part of the proof, we have
\begin{gather}
\l{bv8}
E_\ve(u_\ve, B_{2R}(0))\le I_{\ve, 1, 1}+O(1),
\\
\l{bv9}
E_\ve(v_\ve, B_{R}(0))\ge I_{\ve, 1, 1}+O(1).
\end{gather}

Subtracting the inequalities \eqref{bv8} and \eqref{bv9} and noting that 
\bes
E_\ve(u_\ve, B_{2R}(0))-E_\ve(v_\ve, B_{R}(0))=\frac {3}{16\ve^2}\int_{B_R(0)}(1-|u_\ve|^2)^2=\frac {3}{16\ve^2}\int_{\O}(1-|u_\ve|^2)^2,
\ees
we find that
\be
\l{bv10}
\frac 1{\ve^2}\int_\O (1-|u_\ve|^2)^2\le C_{\color{red}6}(\delta, \O, g). 
\ee

\noindent
{\it Proof  in a general domain  when $C=1$.} We proceed as in the case of a strictly star-shaped domain, using \eqref{bv10} instead of \eqref{as6}.

\smallskip
\noindent
{\it Proof in a general domain for arbitrary $C$.} The inequality $m_\ve\le I_{\ve, 1, C}+O(1)$ is straightforward. On the other hand, we  have (arguing as in the first step and using \eqref{br4}) 
\bes
m_\ve \ge I_{\ve, 1/2, C}+O(1)\ge I_{\ve, 1, C}+O(1).\qedhere
\ees
\ep

\bp[Proof of Theorem \ref{bu8} when $\O$ is strictly-starshaped] We argue by contradiction. Then there exists some $\alpha>0$ such that, passing to a subsequence $\ve_\ell\to 0$ and relabeling the enlarged bad discs if necessary, we have
\begin{gather}
\l{bw2}
\lim_{\ve \to 0}\frac{\ln |x^1_\ve-x^2_\ve|}{\ln \ve}= \alpha,\\
\l{bw3}
\liminf_{\ve\to 0}\frac{\ln |x^i_\ve-x^j_\ve|}{\ln \ve}\ge \alpha, \ \fo i\neq j.
\end{gather}

Note that, by assumption, we have
\be
\l{bw4}
0<\alpha\le\alpha_0<1.
\ee

Possibly after passing to further subsequences, there exist a partition consisting of non-empty sets,
\bes
\{ 1,\ldots, D\}={\mathscr G}_1\sqcup\ldots {\mathscr G}_\ell
\ees 
(with, possibly, $\ell=1$), with each ${\mathscr G}_k$ non-empty,  and a number $0<\beta<\alpha$ such that 
\begin{gather}
\l{bw5}
1, 2 \in {\mathscr G}_1,
\\
\l{bw6}
[i, j\in {\mathscr G}_k, \, i\neq j]\implies \lim_{\ve\to 0}\frac{\ln |x^i_\ve-x^j_\ve|}{\ln \ve}= \alpha,
\\
\l{bw7}
[i\in {\mathscr G}_k, \, j\in {\mathscr G}_n,\, k\neq n]\implies \liminf_{\ve \to 0}\frac{\ln |x^i_\ve-x^j_\ve|}{\ln \ve}\ge \beta.
\end{gather}

Consider now constants $\beta<\alpha_1<\alpha_2<\alpha<\alpha_3<\alpha_4<1$ to be fixed later (in order to obtain a contradiction). We extend $u$ to $\R^2$ as explained earlier in this section. By \eqref{af4},  the upper bound in Lemma \ref{aq1}, and  a mean value argument, there exists a finite constant $C'$ depending on $D$ and on all the above constants such that: for small  $\ve$,  there exist radii $\ve^{\alpha_4}<R_1<\ve^{\alpha_3}<\ve^{\alpha_2}<R_2<\ve^{\alpha_1}$ satisfying
\be
\l{bw8}
R_j\int_{C_{R_j}(x^i_\ve)}|\na u|^2\le C',\  j=1, 2, \, \fo 1\le i\le D.
\ee

Note that (by definition of $\alpha$ and choice of $R_1$), 
\be
\l{bw9}
\text{for small $\ve$, $B_{R_1}(x^i_\ve)\cap B_{R_1}(x^j_\ve)=\emptyset$ if $i\neq j$.}
\ee

For simplicity, assume, only in this paragraph, that $x^i_\ve=0$. From \eqref{bw8} and the fact that $|u|\ge 1/2$ on $C_{R_j}(0)$, $j=1, 2$, we find that, in the annulus $A_{C_1\ve, R_1}$,  $w:=u/|u|$ is a competitor in the class $\tilde {\mathscr H}_{C_1\ve, R_1, C}$ (where $C_1$ is the constant in the definition of the enlarged bad discs and the constant $C$ depends  on $C'$ and on the constant $C_2$ in \eqref{aqb}). 

\smallskip
By the above, we find that
\be
\l{bw10}
E_0(w, B_{R_1}(x^i_\ve)\setminus B_{C_1\ve}(x^i_\ve))\ge  \tilde I_{C_1\ve, R_1, C}-C_3,\ \fo i,
\ee
with $C_3$ a finite constant depending only on the extension of $u$. (Same for the constants $C_4$,..., $C_6$ below.)

\smallskip
Set $D_k:=\# {\mathscr G}_k$, so that
\be
\l{bw11}
D_1\ge 2,\ \sum_k D_k=D,\  \sum_k (D_k)^2\ge D+2.
\ee 

Choose, for each $1\le k\le \ell$, an index $i_k\in {\mathscr G}_k$. Note that
\be
\l{bx1}
\text{for small $\ve$, $B_{R_2}(x^{i_k}_\ve)\cap B_{R_2}(x^{i_j}_\ve)=\emptyset$ if $k\neq j$.}
\ee

We are therefore in position to apply Lemma \ref{bh1} with:
\begin{gather*}
X:=\overline\O,\ U:=\{ x\in\R^2;\, \dist (x, \O)\le 1\},
\\
R:=R_2,\ B_k:=\overline B_{R_2}(x^{i_k}\ve),\ 1\le k\le \ell,
\\
h:=\frac 12|\na w|^2,
\\
\lambda_k:=(D_k)^2, \ 1\le k\le \ell.
\end{gather*}
(The fact that the assumption \eqref{bh2} is satisfied follows from \eqref{bx1} and \eqref{bn2}.) Using Lemma \ref{bh1}, we find that 
\be
\l{bx2}
G_0(w, \O\setminus \cup_k B_{R_2}(x^{i_k}_\ve))\ge \sum_k (D_k)^2 \ln \frac 1{R_2}-C_4.
\ee

Combining \eqref{bx2} with: (i) Lemma \ref{af1} applied in $\O\setminus \cup_k B_{R_2}(x^{i_k}_\ve)$; (ii)  the upper bounds \eqref{bw8} and \eqref{aqa}; (iii) the fact that $|u|\ge 1/2$ on the complement of the enlarged bad discs, we find that
\be
\l{bx3}
E_0(w, \O\setminus \cup_k B_{R_2}(x^{i_k}_\ve))\ge \sum_k (D_k)^2  (1-\delta_3)\ln \frac 1{R_2}-C_5.
\ee

Collecting \eqref{bw9}, \eqref{bw10}, and \eqref{bx3}, we find that
\be
\l{bx4}
E_0(w, \O)\ge D\tilde I_{C_1\ve, R_1, C}+
\sum_k (D_k)^2  (1-\delta_3)\ln \frac1 {R_2}-C_6.
\ee

Using: (i) \eqref{bx4}; (ii) Lemma \ref{bt0}; (iii) Lemma \ref{bu8}; (iv) the fact that $\O$ is strictly star-shaped (and thus \eqref{as6} holds); (v) the  inequalities satisfied by  $R_1$ and $R_2$; (vi) the last inequality in \eqref{bw11},  we find that
\be
\l{bx5}
\ba
E_0(u,\O)\ge &D C_\delta\ln \frac 1\ve+ [(D+2)(1-\delta_3)\alpha_1-DC_\delta\alpha_4]\ln\frac 1\ve \\&+o\left(\ln\frac 1\ve\right)\ \text{as }
\ve\to 0.
\ea
\ee

Since we have $C_\delta\le \pi$ (see Lemma \ref{br0}), and, by assumption,  $\delta_3<2/(D+2)$, we find that
\be
\l{bx6}
(D+2)(1-\delta_3)\alpha_1-DC_\delta\alpha_4>0\ \text{provided $\alpha_1$ and $\alpha_4$ are sufficiently close to $\alpha$.}
\ee

We obtain the desired contradiction {\it via}  \eqref{bx5},  \eqref{bx6}, and the upper bound \eqref{bq6}.
\ep

\bp[Proof of  Theorem \ref{bu5} when $\O$ is strictly star-shaped] Let $0<\alpha_1<\alpha<1$. In what follows, constants are finite and independent of small $\ve$. Consider, for small $\ve$,  a radius  $\ve^{\alpha}<R<\ve^{\alpha_1}$ and a finite constant $C'$ satisfying
\be
\l{bx7}
R\int_{C_{R}(x^i_\ve)}|\na u|^2\le C',\  \fo 1\le i\le D.
\ee

Arguing as for \eqref{bw10}, we have 
\be
\l{bx8}
E_0(w, B_R (x^i_\ve)\setminus B_{C_1\ve}(x^i_\ve))\ge \tilde I_{C_1\ve, R, C}-C_3.
\ee

From \eqref{bx8}, \eqref{bt6}, and the fact that $R>\ve^\alpha$, we find that
\be
\l{bx9}
E_0(w, \O)\ge  \alpha D C_\delta\ln \frac 1\ve+o\left(\ln \frac 1\ve\right).
\ee

The parameter $\alpha$ being arbitrary in $(0,1)$, we obtain from \eqref{bx9} that
\be
\l{bx10}
E_0(w, \O)\ge   D C_\delta\ln \frac 1\ve+o\left(\ln \frac 1\ve\right).
\ee

Since $\O$ is strictly star-shaped, \eqref{bx10}, the upper bound \eqref{as6}, and Lemma \ref{bu9} imply that
\be
\l{bx11}
E_0(u,\O)\ge  D C_\delta\ln \frac 1\ve+o\left(\ln \frac 1\ve\right).
\ee

 On the other hand,  \eqref{bq5} and \eqref{bt6} yield
 \be
 \l{bx12}
 E_\ve(u,\O)\le  D C_\delta\ln \frac 1\ve+o\left(\ln \frac 1\ve\right).
 \ee

 The conclusion then follows from \eqref{bx11} and \eqref{bx12}.  
\ep

\bp[Proof of Theorem \ref{bu8} in a general domain] Arguing as in the proof of \eqref{bv10} and using the upper bound \eqref{bx12} (valid in any domain) and the lower bound \eqref{bx11} (valid in a ball, by the preceding proof), we find that 
\be
\l{by1}
\frac 1{\ve^2}\int_\O (1-|u_\ve|^2)^2=o\left(\ln \frac 1\ve\right).
\ee

We next repeat the proof of Theorem \ref{bu8} in a strictly star-shaped domain. The only difference arises in the justification of \eqref{bx5}: when $\O$ is strictly star-shaped, we use the upper bound \eqref{as6}, while, for a general $\O$, we rely on \eqref{by1}.
\ep

\bp[Proof of Theorem \ref{bu5} in a general domain] As for the preceding proof, we repeat the proof in the strictly star-shaped case, except when it comes to justify \eqref{bx11}, for which we rely on \eqref{by1} instead of \eqref{as6}.
\ep

The proof of Theorem \ref{bw1} relies on the following straightforward variant of Lemma \ref{bu8}, whose proof is left to the reader.

\bl
\l{by2} 
Let $1/2\le a<1$ and $C=C(a)<\infty$ be such that, for the corresponding enlarged bad discs, we have  $|u|\ge a$ in 
 $\omega:=\O\setminus\cup_j B_{C\ve}(x^j_\ve)$. Let $C\ve<R<{\bf m}$.   Then
\be
\l{by3}
\ba
E_0(u, B_R(x^i_\ve)\setminus B_{C\ve}(x^i_\ve))\ge  &E_0(w, B_R(x^i_\ve)\setminus B_{C\ve}(x^i_\ve))
\\
&-C_1(\delta, \O,g) \frac 1{\ve^2}\int_\O (1-|u|^2)^2-C_2(\delta, a, \O, g)
\\
&-C_3(\delta, \O, g)\int_{C_R(x^i_\ve)}|\na u|.
\ea
\ee
\el

\bp[Proof of Theorem \ref{bw1}] Let $0<\alpha_1<\alpha<1$ and let $R$ be as in the proof of Theorem \ref{bu5}. By \eqref{bx7}, \eqref{bx8}, \eqref{bt6}, and \eqref{by3}, we find that 
\be
\l{by4}
E_\ve(u, B_{\ve^\alpha}(x^i_\ve)\setminus B_{C\ve}(x^i_\ve))\ge \alpha C_\delta \ln \frac 1\ve+o\left(\ln \frac 1\ve\right).
\ee

Therefore, for any fixed $r>0$ and for any $J\subset \{ 1, \ldots, D\}$, we have
\be
\l{by5}
E_\ve(u, \cup_{j\in J}B_r(x^i_\ve)\setminus B_{C\ve}(x^i_\ve))\ge \# J C_\delta \ln \frac 1\ve+o\left(\ln \frac 1\ve\right).
\ee

Combining \eqref{by5} with the upper bound \eqref{bx12}, we find that, for any fixed $r>0$,  
\be
\l{by6}
E_\ve (u, \O\setminus \cup_j B_r(a_j))=o\left(\ln \frac 1\ve\right).
\ee

From \eqref{by5} and \eqref{by6}, we obtain that there exist numbers $b_j\ge C_\delta$, $\fo j$, such that,  possibly up to a subsequence,
\be
\l{by7}
\frac{e_\ve(u_\ve)}{\ln (1/\ve)}\rightharpoonup \sum_j b_j\delta_{a_j}\ \ \text{$\ast$-weakly in ${\mathscr M}(\overline\O)$}.
\ee

The fact that $b_j\le C_\delta$, and thus \eqref{bu7}  holds for the full sequence,  is a consequence of \eqref{by7} and \eqref{bx12}.
\ep

\section{Arbitrary \texorpdfstring{$\delta$}{delta} analysis.  Asymptotic expansion of the energy}
\l{ca1}

Throughout this section, we consider minimizers $u=u_\ve$ of $E_\ve$ in $H^1_g(\O; \C)$, with boundary datum of  degree $-D<0$.  
A first main goal is to generalize the formula \eqref{bu6} to any $\delta$ (without any smallness assumption on $\delta$). We will also obtain variants of Theorems \ref{bu8} and \ref{bw1}, under weaker assumption on $\delta$. However, the results below are not necessarily,  strictly speaking, generalizations of the results in the previous section: while they hold either for any $\delta$ or under {\it explicit} smallness conditions on $\delta$,  the picture we get is \enquote{blurred}, in the sense that it involves, instead of enlarged bad discs (as up to now), giant bad discs (that we define below), whose radii can be much bigger than $\ve$.  
(Recall that, in the previous section, we assume that $\delta\le \delta_2=\delta_2(D)$, and the existence of $\delta_2$ is established {\it via} a proof by contradiction. Therefore, the smallness conditions on $\delta$ in the previous section are {\it not explicit}.)  

\smallskip
In order to state the main results of this section, we introduce new notation and several definitions. 

\smallskip
Fix $\delta$. 
Consider the enlarged bad discs constructed in Lemma \ref{at1}. Possibly after passing to a subsequence, we may assume that the number  $M$ of bad discs is independent of $\ve$, and that all the limits
\bes
\lim_{\ve\to 0}\frac{\ln |x^i_\ve-x^j_\ve|}{\ln \ve}:=L_{ij}, \ i\neq j,
\ees
exist. Note that we have $0\le L_{ij}\le 1$, $\fo i\neq j$. There exists a partition consisting of non-empty sets, 
\bes
\{ 1,\ldots, D\}={\mathscr G}^0_1\sqcup\ldots {\mathscr G}^0_{\ell_0}
\ees 
(with, possibly $\ell_0=1$) such that 
\begin{gather}
\l{ca2}
[i, j\in {\mathscr G}^0_k, \, i\neq j]\iff L_{ij}=0,
\\
\l{ca3}
[i\in {\mathscr G}^0_k, \, j\in {\mathscr G}^0_n,\, k\neq n]\iff L_{ij}>0.
\end{gather}

If ${\mathscr G}_k^0$ consists of a single index $i$, the corresponding giant bad disc is simply the enlarged bad disc $B_{C\ve}(x^i_\ve)$. Otherwise, we choose (arbitrarily) $i\in {\mathscr G}_k^0$. We  set 
\bes
R_k=R_k(\ve):=2\min\{ |x^i_\ve-x^j_\ve|;\, j\in {\mathscr G}_k^0,\, j\neq i\}
\ees
and define the giant bad disc associated with ${\mathscr G}^0_k$ as $B_{R_k}(x^i_\ve)$. Note that, while there is some ambiguity in this definition (since it depends on the choice of $i$), the giant bad discs have two common features: (i) for small $\ve$, if $j\in {\mathscr G}^0_k$, then $B_{C\ve}(x^j_\ve)\subset B_{R_k}(x^i_\ve)$; (ii) for small $\ve$, two giant bad discs corresponding to two different ${\mathscr G}^0_k$ are disjoint.

\smallskip
We extend $u$ to $\R^2\setminus \O$ as explained in the previous sections and define the degree $D^0_k$ of the giant bad disc associated with ${\mathscr G}^0_k$ through the formula
\bes
D^0_k:=\deg (u/|u|, C_{R_k}(x^i_\ve)).
\ees

It is straightforward the definition does not depend on the choice of the extension of $u$. 

\smallskip
To give a flavor of the results in this section and how they do compare with the results in the previous sections, we start with a special case of  more general assertions below.

\bt
\l{ca4}
Assume that $\delta\le 2/(D+2)$. Let $0<\alpha<1$. Then, for small $\ve$ (small\-ness depending on $\alpha$), we have
\begin{gather}
\l{ca5}
D^0_k=-1, \ \fo k,\ \text{and thus $\ell_0=D$},
\\
\l{ca6}
m_\ve= D C_\delta \ln\frac 1\ve+o\left(\ln \frac 1\ve\right)\ \text{as }\ve \to 0,
\\
\l{ca7} 
\text{If $x^i_\ve$, $x^j_\ve$ are the centers of two different giant bad discs, then } |x^i_\ve-x^j_\ve|\ge \ve^\alpha,
\\
\l{ca8}
\ba
&\text{If, up to a subsequence, the centers of the giant bad discs satisfy $x^j_\ve\to a_j\in\overline\O$,}
\\
 &\text{$j=1,\ldots, D$, then }
\frac{e_\ve(u_\ve)}{\ln (1/\ve)}\rightharpoonup C_\delta\sum_j\delta_{a_j}\ \ \text{$\ast$-weakly in ${\mathscr M}(\overline\O)$}.
\ea
\end{gather}
\et

This is to be compared respectively with Theorems \ref{bg1}, \ref{bu5}, \ref{bu8}, and \ref{bw1}.

\smallskip
We next introduce a quantity that will play the role of $D C_\delta$ in the general case (i.e., without any size assumption on $\delta$). To start with, let $d\in\Z$ be an integer. Associate with $d$ the classes ${\mathscr H}_{R_1, R_2, C}^d$ and  $\tilde {\mathscr H}_{R_1, R_2, C}^d$, by replacing, in the definition of ${\mathscr H}_{R_1, R_2, C}$ and  $\tilde {\mathscr H}_{R_1, R_2, C}$ (see \eqref{bq2} and \eqref{bt1}),  the condition $\deg (v, C_{R_1}(0))=\deg (v, C_{R_1}(0))=-1$, with the condition $\deg (v, C_{R_1}(0))=\deg (v, C_{R_1}(0))=d$. (In order to have non-empty classes, one has to suppose that $C\ge |d|$.) Consider the corresponding minima $I_{R_1, R_2, C}^d$ and $\tilde I_{R_1, R_2, C}^d$. 
The analysis in Section \ref{bq1} (which corresponds to the special case $d=-1$) can be readily extended to the case of an arbitrary degree condition and yields full analogues of the results in Section \ref{bq1}. We quote e.g., without proof, straightforward 
generalizations of some of the results there.
\begin{gather}
\l{ca9}
d^2 (1-\delta)\pi \ln \frac{R_2}{R_1}-|d|(1-\delta)\pi\le \tilde I_{R_1, R_2, C}\le d^2\pi \ln \frac{R_2}{R_1},
\\
\l{ca10}
\begin{split}
\text{There exists some $d^2 (1-\delta)\pi \le C_\delta^d\le d^2\pi$ such that }\lim_{t\to\infty}\frac{\tilde I_{R_1, tR_1, C}}{\ln t}=C_\delta^d,\\
 \fo R_1>0,\, \fo C\ge |d|.
\end{split}
\end{gather}

Consider now the quantity
\be
\l{cb1}
K(\delta, -D):=\inf\left\{ \sum_{j=1}^M C_\delta^{d_j};\, M\ge 1, d_j\in\Z, \sum_j d_j=-D  \right\}.
\ee

We will see later that it suffices to consider, in \eqref{cb1},  only $M$ and degrees $d_j$ satisfying {\it a priori} bounds depending only on $\delta$ and $D$, and thus, in \eqref{cb1}, 
$\inf$ is actually a  $\min$. We will also see that, under the assumption $\delta\le 2/(D+2)$, we have $K(\delta, -D)=D C_\delta$.

\smallskip
A main result in this section is the following.
\bt
\l{cb2}
We have
\be
\l{cb3}
m_\ve=K(\delta, -D)\ln \frac 1\ve+o\left(\ln \frac 1\ve\right)\ \text{as }\ve \to 0.
\ee
\et

We now proceed to the proofs and establish, on the way, some auxiliary results of independent interest. Since the techniques and arguments used in this section are essentially variants of the ones presented in Sections \ref{bg0}, \ref{bm1}, and \ref{bu1}, the proofs will be rather sketchy and send to similar proofs in  these sections.

\smallskip
We start with a straightforward result.

\bl
\l{cd4}
The infimum in \eqref{cb1} is achieved, and every minimal configuration  $(d_j)_{1\le j\le M}$ such that $d_j\neq 0$, $\fo j$, satisfies $M\le C_1(\delta, D)$, $|d_j|\le C_2(\delta, D)$. 
\el
\bp
This follows from the lower bound $d^2(1-\delta)\pi\le C^d_\delta$ (see \eqref{ca10}).
\ep

\bl
\l{cd1}
For small $\ve$, a giant bad disc has a non-zero degree.
\el
\bp[Sketch of proof]
Proof by contradiction. Suppose that, possibly after a subsequence and relabeling the giant bad discs, we have $D_1^0=0$ and $x^1_\ve\in {\mathscr G}_1^0$. Let $0<\alpha<1$ be such that, for sufficiently small $\ve$,
\be
\l{cd2}
[i\in {\mathscr G}_k^0,\, j\in {\mathscr G}_n^0,\, k\neq n]\implies |x^i_\ve-x^j_\ve|\ge \ve^\alpha.
\ee

Let $\alpha<\beta<1$. Using \eqref{cd2} and the assumption $D_1^0=0$, we are in position to repeat the arguments leading to \eqref{ba4} in the proof of Lemma \ref{ba2}, and find that
\be
\l{cd3}
E_\ve(u, B_{\ve^\beta}(x^1_\ve))\le C_1 \ \text{and}\ G_\ve(B_{\ve^\beta}(x^1_\ve))\le C_1,
\ee
for some finite constant $C_1=C_1(\alpha, \beta, \deg g)$. For small $\ve$, estimate \eqref{cd3}, the fact that $|u(x^1_\ve)|\le 1/2$, and the $\eta$-ellipticity Lemma \ref{ab8} yield a contradiction.
\ep

\bp[Proof of Theorem \ref{cb2}] An argument similar to the one leading to \eqref{bq5} yields  the upper bound
\be
\l{cd5}
m_\ve \le K(\delta, -D)\ln \frac 1\ve+C(\O, g).
\ee

The heart of the proof consists of establishing the lower bound
\be
\l{cd7}
m_\ve \ge K(\delta, -D)\ln \frac 1\ve+o\left(\ln \frac 1\ve\right)\ \text{as }\ve\to 0.
\ee

\noindent
{\it Construction of nested groups of bad discs.} Define $L_{ii}:=1$. Possibly after passing to a further sub\-se\-quence in $\ve$, we may assume that there exist numbers $0\le \alpha_p<\cdots<\alpha_1<\alpha_0= 1$ (with pos\-sibly $p=0$) such that
\be
\l{cd6}
\{ L_{ij}; 1\le i, j\le M\}=\{ \alpha_0,\ldots, \alpha_p\}.
\ee

For $0\le q\le p$, we define the equivalence relation
\bes
i\sim_q j\iff L_{ij}\ge \alpha_q.
\ees

This equivalence relation defines a partition 
\bes
\{ 1,\ldots, D\}={\mathscr G}_1^q\sqcup\ldots\sqcup {\mathscr G}_{\ell_q}^q;
\ees
for $q=0$,  we recover the partition defined at the beginning at this section, and the corresponding equivalence classes define the giant bad discs. Note that these equivalence classes are nested, in the sense that, if $i\sim_q j$ (and thus $i$ and $j$ are in the same equivalence class at the $q$ level), then $i\sim_r j$, $\fo r>q$ (and thus $i$ and $j$ are in the same equivalence class at any higher level). 

\smallskip
\noindent
{\it Proof of \eqref{cd7} when $\O$ is strictly star-shaped}. If $\alpha_p>0$, define $\alpha_{p+1}:=0$; otherwise we do not define $\alpha_{p+1}$. We extend $u$ to $\R^2$ as in the previous sections. For $0\le q\le p-1$ (if $\alpha_p=0$), respectively $0\le q\le p$ (if $\alpha_p>0$), let 
\bes
\alpha_{q+1}<\beta_q'<\beta_q''<\gamma_q'<\gamma_q''<\alpha_q
\ees
be (arbitrary, but fixed at this stage) constants. Consider a radius $R$ such that 
\be
\l{cd8}
\ve^{\gamma_q''}<R<\ve^{\gamma_q'}\text{ or }\ve^{\beta_q''}<R<\ve^{\beta_q'}.
\ee

 Then, for small $\ve$ (smallness depending only on the above constants, not on $R$), 
 \be
 \l{cd9}
 \text{If $i\not\sim_q j$, then $B_R(x^i_\ve)\cap B_R(x^j_\ve)=\emptyset$}.
 \ee
 
Consider now, for each $q$ and $k$, some $i=i(k,q)$ such that $i\in {\mathscr G}_k^q$.   We define the \enquote{ degree of the class ${\mathscr G}_k^q$} as 
 \bes
 D_k^q:=\deg (u/|u|, C_R(x^i_\ve));
\ees
the definition does not depend on the choice of $i$ or of $R$ satisfying \eqref{cd8}, and is independent of the extension of $u$ to $\R^2\setminus \O$. When $q=0$, we recover the definition of the degree of a giant bad disc.
Moreover, by \eqref{cd9} we have
\be
\l{cd10}
\sum_{k} D_k^q=-D,\ \fo q.
\ee

We next choose, using a mean value argument, radii
\be
\l{ce1}
\ve^{\gamma_q''}<\rho^q<\ve^{\gamma_q'}<\ve^{\beta_q''}<R^q<\ve^{\beta_q'}
\ee
such that
\be
\l{ce2}
\rho^q\int_{C_{\rho^q}(x^{i(q,k)}_\ve)}|\na u|^2\le C(D) \text{ and } R^q\int_{C_{R^q}(x^{i(q,k)}_\ve)}|\na u|^2\le C(D),\ \fo q,\, \fo k.
\ee

Let $w:=u/|u|$, well-defined outside the enlarged bad discs. By \eqref{cd9}, if $\rho^q\le R\le R^q$ and $\ve$ is small, then $\deg (w, C_R(x^{i(q,k)}_\ve))=D_k^q$. Combining this fact with \eqref{ce1}, \eqref{ce2}, and \eqref{ca10}, we obtain
\be
\l{ce3}
E_0(w, B_{R^q}(x^{i(q,k)}_\ve)\setminus B_{\rho^q}(x^{i(q,k)}_\ve))\ge  (\gamma_q'-\beta_q'') C_\delta^{D_k^q} \ln \frac 1\ve+o\left(\ln\frac 1\ve\right)\ \text{as }\ve \to 0.
\ee

Summing \eqref{ce3} over $q$ and $k$, and using \eqref{cd9}, \eqref{cd10}, and \eqref{cb1}, and the fact that $u$ is fixed and smooth in $\R^2\setminus\O$, we obtain
\be
\l{ce4}
\ba
E_0(w, \O)\ge &\sum_q (\gamma_q'-\beta_q'') \sum_k  C_\delta^{D_k^q} \ln \frac 1\ve+o\left(\ln\frac 1\ve\right)
\\
\ge & K(\delta, -D)\sum_q (\gamma_q'-\beta_q'')  \ln \frac 1\ve+o\left(\ln\frac 1\ve\right)\ \text{as }\ve\to 0.
\ea
\ee

Letting, in \eqref{ce4}, $\gamma_q'\to \alpha_q$ and $\beta_q''\to \alpha_{q+1}$, and using the fact that $\sum_q (\alpha_q-\alpha_{q+1})=1$,  we find that 
\be
\l{ce5}
E_0(w, \O)\ge  K(\delta, -D)  \ln \frac 1\ve+o\left(\ln\frac 1\ve\right)\ \text{as }\ve\to 0.
\ee

We obtain \eqref{cd7} from \eqref{ce5}, \eqref{as6}, and Lemma \ref{bu8bis}.

\smallskip
\noindent
{\it Proof of \eqref{cd7} in a general domain.} As in the proof of Theorem \ref{bu8}, we rely on the previous step to derive first \eqref{by1}, then \eqref{cd7} in a general domain.
\ep

An inspection of the proof of \eqref{ce4} leads to the following
\bc
\l{ce6}
With the above notation, we have
\be
\l{ce7}
\sum_k C_\delta^{D_k^q}=K(\delta, -D),\ \fo q.
\ee
\ec

\bp[Sketch of proof of Theorem \ref{ca4}] By \eqref{ce7}, we have
\be
\l{ce8}
\sum_k C_\delta^{D_k^0}=K(\delta, -D).
\ee

On the other hand, if $\delta< 2/(D+2)$,  then,  by the first part of \eqref{ca10}, when $d\neq 0$ we have 
\be
\l{cvb1}
C_\delta^d\ge d^2(1-\delta)\pi>\frac D{D+2}d^2\pi.
\ee

Using \eqref{cvb1}, 
we find that
\be
\l{ce9}
\ba
{}[D_k\neq 0,\, \fo k, \ \sum_k D_k=-D]\implies &\text{either }D_k=-1,\, \fo k,
\\
&\text{or }\sum_k (D_k)^2\ge D+2
\text{ and }\sum_k C_\delta^{D_k}>\pi D. 
\ea
\ee

Since, on the other hand, we have (using \eqref{br5})
\be
\l{ce10}
K(\delta, -D)\le D C_\delta\le \pi D,
\ee
we find, from \eqref{ce8}--\eqref{ce10} and Lemma \ref{cd1}, that each giant bad disc is of degree $-1$ (i.e., \eqref{ca5} holds) and that, moreover,
\be
\l{cf1}
K(\delta, -D)=D C_\delta.
\ee

Combining \eqref{cf1} with \eqref{cb3}, we obtain \eqref{ca6} when $\delta<2/(D+2)$. 

\smallskip
When $\delta=2/(D+2)$, we argue similarly (using \eqref{br6} instead of \eqref{br5}), and find that \eqref{ca5} and \eqref{ca6}  still hold in this case.

\smallskip
 On the other hand, by construction, the giant bad discs satisfy the assumption \eqref{bu99}. (With the notation in the proof of Theorem \ref{cb2}, the role of $\alpha_0$ in \eqref{bu99} can be played by any constant $\beta$ with $\alpha_1<\beta<1$.) We are in position to repeat the proof of Theorem \ref{bu8} and obtain, for the centers of giant bad discs and $\ve$ small, the validity of  \eqref{ca7}, which is the analogue of \eqref{bu10}. Combining this with \eqref{cf1}, we are in position to repeat the proof of Theorem  \ref{bw1}, and obtain the validity of \eqref{ca8}.
\ep

 \bibliographystyle{plain}
 \bibliography{bibc}
 
\end{document}